\date{}
\title{
Factorization Norms and Hereditary Discrepancy
}
\newif\ifcmts
\cmtstrue
\newif\ifbigpic
\bigpictrue

\author{
{\sc Ji\v{r}\'{\i} Matou\v{s}ek}\thanks{Research 
supported by the  ERC Advanced Grant No.~267165.}
\\
   {\footnotesize Department of Applied Mathematics}\\[-1.5mm]
   {\footnotesize  Charles University, Malostransk\'{e} n\'{a}m. 25}\\[-1.5mm]
{\footnotesize  118~00~~Praha~1,
   Czech Republic, and}\\%[-1.5mm]
{\footnotesize    Department of Computer Science}\\[-1.5mm]
{\footnotesize    ETH Zurich,
      8092 Zurich, Switzerland}
\and 
{\sc Aleksandar Nikolov}
\\
   {\footnotesize  Microsoft Research}\\[-1.5mm]
   {\footnotesize Redmond, WA, USA}
\and
{\sc Kunal Talwar}
\\
   {\footnotesize Google}\\
   {\footnotesize Mountain View, CA, USA}
}

\documentclass[11pt]{article}
\usepackage{a4}

\usepackage{latexsym}
\usepackage{amssymb,amsmath,amsthm}
\usepackage{graphicx}
\usepackage{url}
\usepackage{color}

\newtheorem{theorem}{Theorem}[section]

\newtheorem{prop}[theorem]{Proposition}

\newtheorem{lemma}[theorem]{Lemma}

\newcommand{\heading}[1]{\vspace{1ex}\par\noindent{\bf\boldmath #1}}

\newcommand{\R}{{\mathbb{R}}}

\newcommand{\N}{{\mathbb{N}}}

\newcommand\eps{\varepsilon}

\newcommand\FF{\mathcal{F}}
\newcommand\GG{\mathcal{G}}

\newcommand\PP{\mathcal{P}}
\newcommand\RR{\mathcal{R}}
\newcommand\CC{\mathcal{C}}
\newcommand\BB{\mathcal{B}}
\newcommand\II{\mathcal{I}}
\renewcommand\AA{\mathcal{A}}
\newcommand{\MM}{\mathcal{M}}
\newcommand\AP{\mathcal{AP}}

\def\:{\colon}

\long\def\onefigure#1#2{%  #1 picture,  #2  caption
\begin{figure*}[tbp]
\begin{center}
#1
\end{center}
\caption{#2}
\end{figure*}
}

\def\immediateFigure#1{%
\smallskip\begin{center}#1\end{center}\smallskip }

\newcommand{\labfig}[2]  % labeled figure
{\onefigure{\mbox{\includegraphics{#1}}}{\label{f:#1} #2} }

\newcommand{\labfigw}[3]  % labeled figure with prescribed width
{\onefigure{\mbox{\includegraphics[width=#2]{#1}}}{\label{f:#1} #3}}

\newcommand{\immfig}[1]  % immediate figure
{\immediateFigure{\mbox{\includegraphics{#1}}}}

\newcommand{\immfigw}[2] % immediate figure with prescribed width
{\immediateFigure{\mbox{\includegraphics[width=#2]{Figures/#1}}}}

\DeclareMathOperator{\herdisc}{herdisc}
\DeclareMathOperator{\disc}{disc}

\DeclareMathOperator{\rdisc}{rdisc}

\DeclareMathOperator{\rank}{rank}
\DeclareMathOperator{\Tr}{Tr}
\DeclareMathOperator{\detlb}{detlb}
\DeclareMathOperator{\diag}{diag}
\newcommand\dd[1]{\,\mathrm{d}#1}

\newcommand{\enorm}[1]{\gamma_2(#1)}

\newcommand{\conv}[1]{C_{#1}} 

\newcommand{\cut}[1]{}

\ifcmts
\newcommand{\marrow}{\marginpar{\boldmath$\longleftarrow$}}
\newcommand{\jirka}[1]{\ifhmode\newline\fi\marrow \textsf{*** (Jirka: ) #1\newline}}
\newcommand{\sasho}[1]{\ifhmode\newline\fi\marrow \textsf{*** (Sasho:
    ) #1\newline}}
\newcommand{\kunal}[1]{\ifhmode\newline\fi\marrow \textsf{*** (Kunal: ) #1\newline}}

\else
\newcommand{\marrow}{}
\newcommand{\jirka}[1]{}

\fi

\begin{document}
\maketitle

\begin{abstract} 
  The $\gamma_2$ norm of a real $m\times n$ matrix $A$ is the minimum
  number $t$ such that the column vectors of $A$ are contained in a
  $0$-centered ellipsoid $E\subseteq\R^m$ which in turn is contained
  in the hypercube $[-t, t]^m$. We prove that this classical quantity
  approximates the \emph{hereditary discrepancy} $\herdisc A$ as
  follows: $\gamma_2(A) = {O(\log m)}\cdot \herdisc A$ and $\herdisc A =
  O(\sqrt{\log m}\,)\cdot\gamma_2(A) $. Since $\gamma_2$ is
  polynomial-time computable, this gives a polynomial-time
  approximation algorithm for hereditary discrepancy. Both
  inequalities are shown to be asymptotically tight.

  We then demonstrate on several examples the power of the $\gamma_2$
  norm as a tool for proving lower and upper bounds in discrepancy
  theory. Most notably, we prove a new lower bound of
  $\Omega(\log^{d-1} n)$ for the \emph{$d$-dimensional Tusn\'ady
    problem}, asking for the combinatorial discrepancy of an $n$-point
  set in $\R^d$ with respect to axis-parallel boxes. For $d>2$, this
  improves the previous best lower bound, which was of order
  approximately $\log^{(d-1)/2}n$, and it comes close to the best
  known upper bound of $O(\log^{d+1/2}n)$, for which we also obtain a
  new, very simple proof. 
\end{abstract}

\section{Introduction}

\heading{Discrepancy and hereditary discrepancy. }
Let $V=[n]:=\{1,2,\ldots,n\}$ be a ground set and 
$\FF=\{F_1,F_2,\ldots,F_m\}$ be a system of subsets of~$V$.
The \emph{discrepancy} of $\FF$ is
\[\disc\FF:=\min_{x\in\{-1,1\}^n} \disc(\FF,x),\]
where the minimum is over all choices
of a vector $x\in\{-1,+1\}^n$ of signs for the points,
and
$\disc(\FF,x):=\max_{i=1,2,\ldots,m}\bigl|\sum_{j\in F_i}x_j\bigr|$.
(A vector $x\in\{-1,1\}^n$ is usually called a \emph{coloring} in this
context.)

This combinatorial notion of discrepancy originated in the classical theory
of \emph{irregularities of distribution}, as treated, e.g.,
in \cite{bc-id-87,drm-ti,abc-gd-97}, and more recently it has found
remarkable applications in computer science and elsewhere
(see \cite{s-tlpm-87,Chazelle-book,m-gd} for general introductions
and, e.g., \cite{larsen14} for a recent use).

For the subsequent discussion, we also need the
notion of discrepancy for matrices: for an $m\times n$
real matrix $A$ we set $\disc A :=\min_{x\in\{-1,1\}^n}\|Ax\|_\infty$,
where $\|\cdot\|_\infty$ is the usual $\ell_\infty$ norm on $\R^m$.
If $A$ is the incidence matrix of the set system $\FF$ as above
(with $a_{ij}=1$ if $j\in F_i$ and $a_{ij}=0$ otherwise),
then the matrix definition coincides with the one for set systems.

We can add elements to the sets in any set system with arbitrarily
large discrepancy to get a new set system with small, even zero,
discrepancy. This phenomenon was
exploited %by Charikar, Newman and Nikolov
in \cite{charikar-al-disc-inapprox} for showing that, assuming
P$\,\ne\,$NP, no polynomial-time algorithm can distinguish systems
$\FF$ with zero discrepancy from those with discrepancy of order
$\sqrt n$ in the regime $m = O(n)$, which practically means that
$\disc\FF$ cannot be approximated at all in polynomial time.

A better behaved notion is the \emph{hereditary discrepancy} of $\FF$,
given by
\[
\herdisc \FF:=\max_{J\subseteq V}\disc(\FF|_J),
\]
were $\FF|_J$ denotes the \emph{restriction} of the set system $\FF$
to the ground set $J$, i.e., $\{F\cap J:F\in\FF\}$. Similarly,
for a matrix $A$, $\herdisc A := \max_{J\subseteq[n]}{\disc
  A_J}$ where $A_J$ is the submatrix of $A$ consisting of the
columns indexed by the set~$J$.
  
At first sight, hereditary discrepancy may seem harder to deal with
than discrepancy. For example, while $\disc\FF\le k$
has an obvious polynomial-time verifiable certificate, namely, a suitable
coloring $x\in\{-1,1\}^n$, it is not at all clear how one could certify
either $\herdisc\FF\le k$ or $\herdisc\FF>k$ in polynomial time.

Nevertheless, hereditary discrepancy has turned out to have
significant advantages over discrepancy.  Most of the classical upper
bounds for discrepancy of various set systems actually apply to
hereditary discrepancy as well.  The \emph{determinant lower bound}, a
powerful tool introduced by Lov\'asz, Spencer and Vesztergombi
\cite{lsv-dssm-86}, works for hereditary discrepancy and \emph{not}
for discrepancy.  The determinant lower bound for a matrix $A$ is the
following algebraically defined quantity:
\[\detlb A=
\max_k\max_B |\det B\,|^{1/k},
\]
 where $B$ ranges over all $k\times k$
submatrices of~$A$. Lov\'asz et al.\ proved 
that $\herdisc A\ge\frac12\detlb A$ for all
$A$. Later it was shown in  \cite{almosttight} that $\detlb A$
also bounds $\herdisc A$ from above up to a polylogarithmic factor;
namely, $\herdisc A=O(\log(mn)\sqrt{\log n}\,)\cdot\detlb(A)$.

While the quantity $\detlb A$ enjoys some pleasant properties,
there is no known polynomial-time algorithm for computing it.
Bansal \cite{bansal-di} provided a polynomial-time algorithm that, given
a system $\FF$ with $\herdisc \FF\le D$, computes a coloring $x$
witnessing $\disc\FF=O(D\log(mn))$. However, this is not an approximation
algorithm for the hereditary discrepancy in the usual sense,
since it may find a low-discrepancy coloring even for $\FF$ with large
hereditary discrepancy.

\heading{The $\gamma_2$ factorization norm. } The first
polynomial-time approximation algorithm with a polylogarithmic
approximation factor for hereditary discrepancy was found by the last
two authors and Zhang \cite{NTZ-diffprivacy}. Here we strengthen and
streamline this result, and show that hereditary discrepancy is
approximated by the $\gamma_2$ factorization norm from Banach space
theory. (This connection was implicit in~\cite{NTZ-diffprivacy}.)  A
preliminary version of our results, which we have simplified and
extended, appeared in the conference
publications~\cite{NT,e8-tusnady}. For some of the simplifications we
are indebted to Noga Alon and Assaf Naor, who pointed out that the
geometric quantity used in~\cite{NT,e8-tusnady} is in fact equivalent
to the $\gamma_2$ norm.

The
$\gamma_2$ norm of an $m\times n$ matrix $A$, taken as a linear operator from $\ell_1^n$ to
$\ell_\infty^m$, is defined as
\begin{equation*}%\label{e:gamma2-defn}
\gamma_2(A) := \min\{\|B\|_{2\to\infty}\|C\|_{1\to 2}: A = BC\}.
\end{equation*}
Above, $\|\cdot\|_{p\to q}$ stands for the $\ell_p \to \ell_q$
operator norm, and $B:\ell_1^n \to \ell_2$, $C:\ell_2\to
\ell_\infty^m$ range over linear operators. Without loss of
generality, we can assume that the rank of $B$ and $C$ is at most
 the rank of $A$. Treating $B$ and $C$ as matrices, it is easy
to see that $\|B\|_{2\to\infty}$ is equal to the largest Euclidean
norm of row vectors of $B$, and $\|C\|_{1\to 2}$ is equal to the
largest Euclidean norm of column vectors of $C$. Moreover, by a
standard compactness argument, the minimum is achieved in this
finite-dimensional case.

We will also make use of an equivalent geometric definition of
$\enorm{A}$. Let the $\ell_\infty$ norm $\|E\|_\infty$ of an ellipsoid
$E$ be defined as the largest $\ell_\infty$ norm of any point in
$E$. Then $\enorm{A}$ is equal to the minimum $\ell_\infty$ norm of a
$0$-centered ellipsoid $E$ that contains all column vectors of $A$, as
is illustrated in the next picture (for $m=2$): 
\immfig{einf}

Given a factorization $A=BC$ witnessing $\gamma_2(A)$, an optimal
ellipsoid can be defined as $\{Bx: \|x\|_2 \leq \|C\|_{1\to 2}\}$. In
the reverse direction, given a 0-centered ellipsoid $E = \{x: x^T M x
\leq 1\}$, defined by a positive definite matrix $M$ with positive
square root $M^{1/2}$, and containing the columns of $A$, the
factorization $A = M^{-1/2} (M^{1/2}A)$ satisfies
$\|M^{-1/2}\|_{2\to\infty} = \|E\|_\infty$ and $\|M^{1/2}A\|_{1\to 2}
\leq 1$.

%\sasho{Do we need the above paragraph? More explanation? Move to section 2?}

We use the notation $\enorm \FF$ for a set system $\FF$ to mean the
$\gamma_2$ norm of the incidence matrix of $\FF$.

\heading{Results on the $\gamma_2$ norm. }  A number of useful
properties of $\gamma_2$ are known, such as the non-obvious fact that
it is indeed a norm~\cite{Tomczak-J-book} (we give an example of how
the triangle inequality fails for $\detlb$), and the fact that it is
is multiplicative under the \emph{Kronecker product} (or tensor
product) of matrices~\cite{LeeSS08}. We further prove a stronger form
of the triangle inequality for matrices supported on disjoint subsets
of the columns. 

\heading{Relationship between $\gamma_2$ and $\herdisc$. } 
Next we prove the following two inequalities relating $\enorm A$ and
$\herdisc A$, which are central to our work: there exists a constant
$C$ such that for every matrix $A$ with $m$ rows,
\begin{align}
  &\herdisc A\ge \frac {\enorm A} {C\log m}, \mbox{
    and} \label{e:NTineq1}\\
  &\herdisc A \le \enorm A \cdot C\sqrt{\log m}\label{e:NTineq2}
\end{align}
(As we will see in Section~\ref{s:simplified-ineq1} below,
\eqref{e:NTineq1} is actually valid with $\log\rank A$ instead of
$\log m$.)  Moreover, $\enorm A$ can be approximated to any desired
accuracy in polynomial time using semidefinite
programming~\cite{LinialMSS07-signmatrices}.  These results together
provide an $O(\log^{3/2}m)$-approximation algorithm for $\herdisc A$,
improving on the  $O(\log^3 m)$-approximation from~\cite{NTZ-diffprivacy}.
%unlike \eqref{e:NTineq2}---see Theorem~\ref{t:biggerherdisc}.)

The lower bound~\eqref{e:NTineq1} is proved using a dual
characterization of $\gamma_2(A)$ in terms of the trace
norm~\cite{LeeSS08}, which we relate to $\detlb A$. Our proof also
implies that $\enorm A$ is between $\detlb A$ and $O(\log
m) \cdot\detlb(A)$.  The upper bound~\eqref{e:NTineq2} is proved using a result of
Banaszczyk~\cite{banasz98}. It is not constructive, in the sense that
we do not know of a polynomial-time algorithm that computes a coloring
achieving the upper bound. Nevertheless, the algorithms of
Bansal~\cite{bansal-di} or Rothvoss~\cite{Rothvoss14-convex} can be
used to find colorings with discrepancy $O(\log m)\cdot \enorm A$ in
polynomial time.

We show that both inequalities \eqref{e:NTineq1} and \eqref{e:NTineq2}
are asymptotically tight in the worst case.
% The tightness of \eqref{e:NTineq2} is proved using almost-spherical
% slices of a high-di\-mension\-al cube.
For \eqref{e:NTineq1}, the asymptotic tightness is demonstrated
on the following simple example: 
for the system $\II_n$ of initial segments of $\{1,2,\ldots,n\}$,
whose incidence matrix is the lower triangular matrix $T_n$ with
$1$s on the main diagonal and below it, we prove that 
the $\gamma_2$ norm is of order $\log n$,
while the hereditary discrepancy is well known to be~$1$. It is
interesting to compare our bounds on $\gamma_2(T_n)$ with a related
but incomparable result of Fredman~\cite{Fredman82}, who showed that
in any factorization $T_n = AB$, with $A$ and $B$ matrices over the integers, the smallest achievable total number of non-zero entries in $A$
and $B$ is $2n\log_\lambda n + O(n)$ for $\lambda=3+2\sqrt{2}$.

We have computed optimal ellipsoids witnessing $\enorm{T_n}$
numerically for moderate values of $n$, and they display a remarkable
and aesthetically pleasing ``limit shape''. It would be interesting
to understand these optimal ellipsoids theoretically---we leave this
as an open problem.

\heading{Applications in discrepancy theory. } In the second part of
the paper we apply the $\gamma_2$ norm to prove new results
on combinatorial discrepancy, as well as to give simple new proofs of known
results.

The most significant result is a new lower bound for the
$d$-dimensional Tusn\'ady's problem; before stating it,
let us give some background.

\heading{The ``great open problem.'' }
Discrepancy theory started with a result conjectured by van der Corput
\cite{c-vI-35,c-vII-35} and first proved by van Aardenne-Ehrenfest 
\cite{a-pijd-45,a-ijd-49}, stating that every infinite sequence
$(u_1,u_2,\ldots)$ of real numbers in $[0,1]$ must have a significant
deviation from a ``perfectly uniform'' distribution. Roth \cite{r-id-54}
found a simpler proof of a stronger bound, and he re-cast the
problem in the following setting, dealing with finite point
sets in the unit square $[0,1]^2$ instead of infinite sequences in $[0,1]$:

Given an $n$-point set $P\subset [0,1]^2$,
the \emph{discrepancy} of $P$ is defined as
\[
D(P,\RR_2):=\sup\Bigl \{\Bigl||P\cap R|-n\lambda^2(R\cap[0,1]^d) \Bigr|: R\in\RR_2\Bigr\},
\]
where $\RR_2$ denotes the set of all $2$-dimensional axis-parallel
rectangles (or $2$-dimensional
intervals), of the form $R=[a_1,b_1]\times [a_2,b_2]$, 
and $\lambda^2$ is the area (2-dimensional Lebesgue measure).
More precisely, $D(P,\RR_2)$ is the \emph{Lebesgue-measure discrepancy} of $P$
w.r.t.\ axis-parallel rectangles. Further let $D(n,\RR_2)=\inf_{P:|P|=n} D(P,\RR_2)$
be the best possible discrepancy of an $n$-point set.

Roth proved that $D(n,\RR_2)=\Omega(\sqrt{\log n})$, while earlier work
of van der Corput yields $D(n,\RR_2)=O(\log n)$. Later Schmidt 
\cite{s-idVII-72} improved the lower bound to $\Omega(\log n)$.

Roth's setting immediately raises the question about a higher-dimensional
analog of the problem: letting $\RR_d$ stand for the system of all
axis-parallel boxes (or $d$-dimensional intervals) in $[0,1]^d$, 
what is the order
of magnitude of $D(n,\RR_d)$? There are many ways of showing an upper
bound of $O(\log^{d-1}n)$, the first one being the Halton--Hammersley
construction \cite{h-mcmsm-60,h-ecqsp-60}, while Roth's lower bound
method yields $D(n,\RR_d)=\Omega(\log^{(d-1)/2}n)$.
In these bounds, $d$ is considered
fixed and the implicit constants in the $O(.)$ and $\Omega(.)$ notation
may depend on it.

Now, over 50 years later, the upper bound is still the best known,
and Roth's lower bound has been improved only a little: first for $d=3$
by Beck \cite{b-tdaet-89} and by Bilyk and Lacey \cite{bl-sbi3D},
and then for all $d$ by Bilyk, Lacey, and Vagharshakyan \cite{blv-sbi}.
The lower bound from \cite{blv-sbi} has the form $\Omega((\log n)^{(d-1)/2+\eta(d)})$, where $\eta(d)>0$ is a constant depending on $d$, with
$\eta(d)\ge c/d^2$ for an absolute constant $c>0$.
Thus, the upper bound for $d\ge 3$ is still about the square of the lower
bound, and closing this significant gap  is called
the ``great open problem'' in the book \cite{bc-id-87}.

\heading{Tusn\'ady's problem. } 
Here we essentially solve a combinatorial
analog of this problem.
In the 1980s Tusn\'ady raised a question
which, in our terminology, can be stated as follows. Let $P\subset\R^2$
be an $n$-point set, and let $\RR_2(P):=\{R\cap P:
R\in\RR_2\}$ be the system of all subsets of $P$ induced by
axis-parallel rectangles $R\in\RR_2$. What can be said about
the discrepancy of such a set system for the worst possible $n$-point $P$?
In other words, what is \[
\disc(n,\RR_2)=\max\{\disc\RR_2(P):
|P|=n\}?
\]

 We stress that for the Lebesgue-measure discrepancy $D(n,\RR_d)$
 we ask for the best placement of $n$ points so that each rectangle
 contains approximately the right number of points, while for 
 $\disc(n,\RR_2)$ the point set $P$ is given by an adversary,
 and we seek a $\pm1$ coloring so that the points in each rectangle
 are approximately balanced.

 Tusn\'ady actually asked if $\disc(n,\RR_2)$ could be bounded by a
 constant independent of~$n$.  This was answered negatively by Beck
 \cite{b-btcfs-81}, who also proved an upper bound of $O(\log^4 n)$.
 His lower bound argument uses a ``transference principle,'' showing
 that the function $\disc(n,\RR_2)$ in Tusn\'ady's problem cannot be
 asymptotically smaller than the smallest achievable Lebesgue-measure
 discrepancy of $n$ points with respect to axis-aligned boxes.  (This
 principle is actually simple to prove and quite general; Simonovits
 attributes the idea to V.\,T.\,S\'os. The main observation is that
 for any coloring with discrepancy $D_1$ of an $n$-point set with
 Lebesgue measure discrepancy $D_2$, the smaller of the two color
 classes has Lebesgue measure discrepancy at most $\frac12 (D_1 +
 D_2)$.) The upper bound was improved to $O((\log n)^{3.5+\eps})$ by
 Beck \cite{b-btcfs-89}, to $O(\log^3 n)$ by Bohus \cite{b-odtp-90},
 and to the current best bound of $O(\log^{2.5} n)$ by Srinivasan
 \cite{s-idbsm-96}.

The obvious $d$-dimensional generalization of Tusn\'ady's problem was
attacked by similar methods. All known lower bounds so far relied on
the transference principle mentioned above. The current best upper
bound for $d\ge 3$ is $O(\log^{d+1/2}n)$ due to Larsen
\cite{larsen14}, which is a a slight strengthening of a previous bound
of $O(\log^{d+1/2} n\sqrt{\log\log n}\,)$ from~\cite{m-dipol-97}.

Here we improve on the lower bound for the $d$-dimensional Tusn\'ady's problem
significantly; while up until now the uncertainty in the exponent of
$\log n$ was roughly between $(d-1)/2$ and $d+1/2$, we reduce it to
$d-1$ versus $d+1/2$. 

\begin{theorem}\label{t:tusnady}
For every fixed $d\ge 2$ and for infinitely many values of $n$,
there exists an $n$-point set $P\subset\R^d$ with
\[
\disc \RR_d(P)=\Omega(\log^{d-1}n),
\]
where the constant of proportionality depends only on~$d$.
\end{theorem}

%\sasho{I guess this depends on whether we want to say something about
%  higher dimension. I am inclined to say this level of detail is fine
%  for the intro.}

From the point of view of the ``great open problem,'' this result is
perhaps somewhat disappointing, since it shows that, in order to
determine the asymptotics of the Lebesgue-measure discrepancy
$D(n,\RR_d)$, one has to use some special properties of the Lebesgue
measure---combinatorial discrepancy cannot help, at least for
improving the upper bound.
 In Section~\ref{s:Lpbound} we will discuss
 a bound on \emph{average} discrepancy, which in a sense separates the
 combinatorial discrepancy (as in Tusn\'ady's problem) from the
 Lebesgue-measure discrepancy.

Using the $\gamma_2$ norm as the main tool, our proof
of Theorem~\ref{t:tusnady} is surprisingly simple. In a nutshell,
first we observe that, since the target bound is polylogarithmic in $n$,
instead of estimating the discrepancy for some cleverly constructed
 $n$-point set $P$,  we can bound from below the hereditary discrepancy 
of the regular $d$-dimensional grid $[n]^d$, where $[n]=\{1,2,\ldots,n\}$.
By a standard and well known reduction, instead of all $d$-dimensional
intervals in $\RR_d$, it suffices to consider only ``anchored''
intervals, of the form $[0,b_1]\times\cdots\times[0,b_d]$.
Now the main observation is that the set system $\GG_{d,n}$
induced on $[n]^d$ by anchored intervals is a $d$-fold product of the system 
$\II_n$ of one-dimensional
intervals mentioned earlier, and its incidence matrix is
the $d$-fold Kronecker product of the matrix~$T_n$.
Thus, by the properties of the $\gamma_2$ norm, we
get that $\enorm{\GG_{d,n}}$ is of order $\log^d n$, and
inequality \eqref{e:NTineq1} finishes the proof of Theorem~\ref{t:tusnady}.

At the same time, using the other inequality \eqref{e:NTineq2},
we obtain a new proof of the best known
upper bound $\disc(n,\RR_d)=O(\log^{d+1/2}n)$, with no extra effort.
This proof is very different from the previously known ones and
relatively simple.

The same method also gives a surprisingly precise upper bound
on the discrepancy of the set system of all subcubes of the
$d$-dimensional cube $\{0,1\}^d$, where this time $d$ is a variable
parameter, not a constant as before. This discrepancy has
previously been studied in \cite{chl-tbhd,chaz-lvov-boxes,NT-HAP},
and it was known that it is between $2^{c_1d}$ and $2^{c_2d}$ for
some constants $c_2>c_1>0$. In Section~\ref{s:hibox} we
show that it is $2^{(c_0+o(1))d}$, for $c_0=\log_2(2/\sqrt 3)\approx 0.2075$.

\heading{General theorems on discrepancy. } Transferring the various
properties of the $\gamma_2$ norm into the setting of
hereditary discrepancy via inequalities \eqref{e:NTineq1},
\eqref{e:NTineq2}, we obtain general results about the behavior of
discrepancy under operations on set systems.
In particular, we get a sharper version of a result of
\cite{almosttight} concerning the discrepancy of the union
of several set systems, and a new bound on the discrepancy of
a set system $\FF$ in which every set $F\in\FF$ is a disjoint
union $F_1\cup\cdots\cup F_t$, where $\FF_1,\ldots,\FF_t$
are given set systems and $F_i\in\FF_i$, $i=1,2,\ldots,t$.
These consequences are presented in Section~\ref{s:disc-thms},
together with some examples showing them to be quantitatively near-tight.

\heading{Other problems in combinatorial discrepancy: new simple
  proofs.}  In Section~\ref{s:sbounds} we revisit two set systems
for which discrepancy has been studied extensively: arithmetic
progressions in $[n]$ and intervals in $k$ permutations of $[n]$. In
both of these cases, asymptotically tight bounds have been known.
Using the $\gamma_2$ norm we recover almost tight upper bounds, up to
a factor of $\sqrt{\log n}$, with very short proofs.

\heading{Immediate applications in computer science.}  Our lower bound
for Tusn\'ady's problem implies a lower bound of $\sqrt{t_ut_q} =
\Omega(\log^{d}n)$ on the update time $t_u$ and query time $t_q$ of
constant multiplicity oblivious data structures for orthogonal range
searching in $\R^d$ in the group model.  This is tight up to a
constant, and strengthens a prior result of Larsen, who showed
$\sqrt{t_ut_q} \geq \log^{(d-1)/2} n$~\cite{larsen14}. Our lower bound
is incomparable with the results of Fredman~\cite{Fredman82}, who
proved the lower bound $(t_u + t_q)/2 = \Omega(\log n)$ only for $d=1$
but in a stronger model that makes no assumption on multiplicity. The
relationship between hereditary discrepancy and differential privacy
from~\cite{MN-stoc12} and the lower bound for Tusn\'ady's problem
imply that the necessary error for computing orthogonal range counting
queries under differential privacy is $\Omega(\log^{d-1}n)$, which is
best possible up to a factor of $\log n$.

Our lower and upper bounds on the discrepancy of
subcubes of the Boolean cube $\{0,1\}^d$ and the results
from~\cite{NTZ-diffprivacy} imply that the necessary and sufficient
error for computing marginal queries on $d$-attribute databases under
differential privacy is $(2/\sqrt{3})^{d+o(d)}$.

\heading{Discrepancy in communication complexity.} A notion that is also known
as discrepancy, but distinct from combinatorial or hereditary
discrepancy, is a standard tool for proving lower bounds in communication complexity. It is commonly
defined for an $m\times n$ matrix $A$ with entries in $\{-1, 1\}$ as
\[
\rdisc A := \min_P \max_{I,J} \Bigl|\sum_{i \in I, j \in J}{p_{ij}a_{ij}}\Bigr|,
\]
where $P$ ranges over $m\times n$ matrices with non-negative entries
such that $\sum{p_{ij}} = 1$, $I$ ranges over subsets of the rows of
$A$, and $J$ ranges over subsets of the columns. To distinguish this
notion from $\disc$, we call it ``rectangle discrepancy''. Linial and
Shraibman~\cite{LinialS09-learning} related $\rdisc$ to $\gamma_2$:
they proved that $(\rdisc A)^{-1}$ is equal, up to constant factors,
to $\min_B \gamma_2(B)$, where $B$ ranges $m\times n$ real matrices
satisfying $a_{ij}b_{ij} \geq 1$ for all $i$ and $j$. Together with
our results, this implies that there exists an absolute constant $C$
so that for any $m\times n$ matrix $A$ with entries in $\{-1, 1\}$
\[
C\sqrt{\log m} \geq (\rdisc A)(\min_B\herdisc B) \geq \frac{1}{C\log
  m},
\]
with the minimum taken over matrices $B$ as above. This is the first formal
connection between hereditary discrepancy and rectangle discrepancy
that we are aware of. 

The
papers~\cite{LinialMSS07-signmatrices,LinialS09-CC,LinialS09-learning}
further connect the $\gamma_2$ norm to various other complexity
measures of sign matrices, in particular the margin and dimension
complexity from learning theory, and randomized and quantum
communication complexity. Using \eqref{e:NTineq1} and \eqref{e:NTineq2},
we can replace $\gamma_2$ with $\herdisc$ in each of these results, at
the cost of losing polylogarithmic factors in the bounds.

\section{Properties of the $\gamma_2$  norm}\label{s:props}

The $\gamma_2$ norm has various favorable properties, which make it a
very convenient and powerful tool in studying hereditary discrepancy,
as we will illustrate later on. We begin by recalling some classical
facts.

\subsection{Known properties of $\gamma_2$}

Observe that the norm $\|B\|_{ 2\to\infty}$ is monotone non-increasing
under removing rows of $B$. Similarly, $\|C\|_{1\to 2}$ is monotone
non-increasing under removing columns of $C$. It then follows  that $\gamma_2(A)$ is monotone
non-increasing under taking an arbitrary submatrix of $A$: for any
subset $I$ of the rows and any subset $J$ of the columns we have
\begin{equation}\label{e:monotone}
\gamma_2(A_{I,J}) \leq \gamma_2(A),
\end{equation}
where $A_{I,J}$ is the submatrix of $A$ induced by $I$ and $J$. This
trivial observation turns out to be crucial for relating $\gamma_2$ and
hereditary discrepancy.

Next we observe that $\gamma_2$ is invariant under transposition. This
is surely a well-known fact, but we give the short proof for
completeness.
\begin{lemma}\label{l:transp} $\enorm A=\enorm{A^T}$.
\end{lemma}

\begin{proof} 
  Let $A = B_0C_0$ be a factorization that achieves $\gamma_2(A)$,
  i.e.~$\gamma_2(A) = \|B_0\|_{2\to\infty}\|C_0\|_{1\to 2}$.  Since
  $A^T = C_0^T B_0^T$, $\|C_0^T\|_{2\to \infty} = \|C_0\|_{1\to 2}$
  and $\|B_0^T\|_{1 \to 2} = \|B_0\|_{2\to \infty}$, we have
  \[
  \gamma_2(A^T) \leq \|C_0^T\|_{2\to \infty} \|B_0^T\|_{1\to 2} 
  = \|C_0\|_{1\to 2} \|B_0\|_{2\to\infty}= \gamma_2(A).
  \] 
  The reverse inequality $\gamma_2(A) \leq \gamma_2(A^T)$ follows by
  symmetry.
\end{proof}

The next (non-obvious) fact implies that $\gamma_2$ is indeed a
norm. For a proof see e.g.~\cite{Tomczak-J-book}.
\begin{prop}[Triangle inequality]\label{p:mnorm}
We have $\enorm{A+B}\le\enorm A+\enorm B$ for every two $m\times n$
real matrices $A,B$.
\end{prop}

\heading{Remark on the determinant lower bound. }
Here is an example showing that the determinant lower
bound of Lov\'asz et al.\ does not satisfy the (exact)
triangle inequality:
for
\[
A=\begin{pmatrix}1&1\\0&1\end{pmatrix}, \ \ 
B=\begin{pmatrix}1&0\\-1&1\end{pmatrix},
\]
we have $\detlb A=\detlb B=1$, but $\detlb (A+B)=\sqrt 5$.

It may still be that the determinant lower bound satisfies
an approximate triangle inequality, say in the following sense:
$\detlb (A_1+\cdots+A_t)\stackrel{?}{\le} O(t)\cdot \max_i \detlb A_i$.
However, at present we can only prove this kind of inequality
with $O(t^{3/2})$ instead of~$O(t)$. 
%and hence we cannot obtain
%an analog of Corollary~\ref{c:disjpieces} using the determinant
%lower bound instead of the ellipsoid-infinity norm.

\heading{Kronecker Product.}
Let $A$ be an $m\times n$ matrix and $B$ a $p\times q$ matrix.
We recall that the \emph{Kronecker product} $A\otimes B$ is 
the following $mp\times nq$ matrix, consisting of $m\times n$
blocks of size $p\times q$ each:
\[
\begin{pmatrix} a_{11}B& a_{12}B&\ldots& a_{1n}B\\ 
                 \vdots &\vdots&\vdots&\vdots\\
                a_{m1}B& a_{m2}B&\ldots&a_{mn}B
\end{pmatrix}
\]

In~\cite{LeeSS08} it was shown that $\gamma_2$ is multiplicative with
respect to the Kronecker product:

\begin{theorem}[{\cite[Thm.~17]{LeeSS08}}]\label{t:krone}
 For every two matrices $A,B$ we have
\[
\enorm {A\otimes B}=\enorm A\cdot\enorm B.
\]
\end{theorem}

\heading{Semidefinite and dual formulations.}
We recall a formulation of $\enorm A$ as a semidefinite program. For
matrices with entries in $\{-1, 1\}$ this program was given
in~\cite{LinialMSS07-signmatrices}; the (easy) generalization to
general matrices can be found in~\cite{LeeSS08}. We have
\[
\begin{array}{rll}
  \enorm{A} = \min t&\mbox{s.t. }\\
  &X_{ii} \leq t& i=1, \ldots, m+n,\\
  &X_{i,m+j} = a_{ij}& i= 1, \ldots, m, j = 1, \ldots, n,\\
  &X \succeq 0
\end{array}
\]

Using standard techniques in convex optimization, e.g.~the ellipsoid
algorithm~\cite{gls-gaco-88}, the program above can be solved to any
given degree of accuracy in time polynomial in $m$, $n$, and the bit
representation of $A$. This gives a polynomial time algorithm to
approximate $\gamma_2(A)$ arbitrarily well.

Using the semidefinite formulation, and the duality theory for
semidefinite programming, Lee, Shraibman and \v{S}palek~\cite{LeeSS08}
derived a dual characterization of the $\gamma_2$ norm as a
maximization problem. This  characterization is a basic tool for
bounding $\gamma_2$ from below.  Let $\|A\|_*$ denote the \emph{nuclear
  norm} of a matrix $A$, which is the sum of the singular values of
$A$ (other names for $\|A\|_*$ are \emph{Schatten $1$-norm},
\emph{trace norm}, or \emph{Ky Fan $n$-norm}; see the text by
Bhatia~\cite{Bhatia-MA} for general background on symmetric matrix
norms).
%\jirka{Refer to some basic source on the properties of the nuclear norm?}

\begin{theorem}[{\cite[Thm.~9]{LeeSS08}}]\label{t:e8dual}
We have
\[
\enorm A=\max\{\|P^{1/2}AQ^{1/2}\|_*: P,Q \mbox{ diagonal},
\mbox{nonnegative}, \Tr P=\Tr Q=1\}.
\]
\end{theorem}

In particular,
several times we will use this theorem with $A$ a square matrix and
$P=Q=\frac 1n I_n$, in which case it gives $\enorm A\ge
\frac1n\|A\|_*$.

\subsection{Putting matrices side-by-side}

We can strengthen the triangle inequality for $\gamma_2$ when the
matrices have disjoint supports. 

\heading{On ellipsoids. }
An ellipsoid $E$ in $\R^m$
is often defined as $\{x\in\R^m: x^TMx\le 1\}$, where
$M$ is a positive definite matrix. Here we will mostly work with
the \emph{dual matrix} $D=M^{-1}$. Using this dual matrix
we have (see, e.g., \cite{combinElli})
\begin{equation}\label{e:duel} %\[
E =E(D)= \{ z\in \R^m : z^Tx \leq \sqrt{x^TDx} \mbox{ for all } x\in \R^m\}.
\end{equation}
This definition can also be used for $D$ only positive semidefinite;
if $D$ is singular, then $E(D)$ is a flat (lower-dimensional)
ellipsoid.

We will use the following formula for $\|E(D)\|_\infty$:
\begin{lemma}\label{l:dual-e8}
  For any $m\times m$ positive semidefinite matrix $D$,
  $\|E(D)\|_\infty = \max_i \sqrt{d_{ii}}$.
\end{lemma}
\begin{proof}
  Let $t:= \max_i \sqrt{d_{ii}}$, and let $e_i$ be the $i$-th standard
  basis vector of $\R^m$. By the definition of $E(D)$, we have that
  $\forall z \in E(D): z_i = z^T e_i \leq \sqrt{e_i^T De_i} =
  \sqrt{d_{ii}}$, and, similarly, $-z_i \leq \sqrt{d_{ii}}$. This
  implies that for any $z \in E(D)$, $\|z\|_\infty \leq t$, and,
  therefore, $\|E(D)\|_\infty \leq t$. Next we show that there exists
  a point $z \in E(D)$ such that $\|z\|_\infty \geq t$, which implies
  $\|E(D)\|_\infty \geq t$ as well. Let $i_0$ be such that
  $\sqrt{d_{i_0, i_0}} = t$, and define $z:= De_{i_0}/t$. Then, by the
  Cauchy-Schwarz inequality
  \[
  \forall x \in \R^m: z^T x = \frac{1}{t} e_{i_0}^T D x \leq
\frac1t \sqrt{(e_{i_0}^TDe_{i_0})(x^TDx)} = \sqrt{x^TDx},
  \]
  so $z \in E(D)$. Moreover, $\|z\|_\infty \geq z_{i_0} = d_{i_0,i_0}/t = t$.
\end{proof}

\begin{lemma}\label{l:union}
 Let $A,B$ be matrices, each with $m$ rows, and let
$C$ be a matrix in which each column is a column of $A$ or of $B$.
Then
\[
\enorm C^2\le\enorm A^2+\enorm B^2.
\]
\end{lemma}

\begin{proof} After possibly reordering
the columns of $C$, we can write $C=\tilde A+\tilde B$, where the first $k$ columns
of $\tilde A$ are among the columns of $A$ and the remaining $\ell$ columns are
zeros, and the last $\ell$ columns of $\tilde B$ are among the columns of $B$
and the first $k$ are zeros.
By~\eqref{e:monotone}, $a:=\enorm {\tilde A}\le \enorm A$,
$b:=\enorm{\tilde B}\le\enorm B$.

We will work with the geometric definition of $\gamma_2$. Let
$E_1=E(D_1)$ and $E_2=E(D_2)$ be ellipsoids witnessing $\enorm{\tilde
  A}$ and $\enorm{\tilde B}$, respectively.  We claim that the
ellipsoid $E(D_1+D_2)$ contains all columns of $A$ and also all
columns of $B$. This is clear from the definition of the ellipsoid
$E(D)=\{z:z^Tx\le\sqrt{x^TDx}\mbox{ for all }x\}$, since for every
$x$, we have \[x^T(D_1+D_2)x=x^TD_1x+x^TD_2x\ge \max\{ x^TD_1x, x^TD_2x\}\] by the positive
semidefiniteness of $D_1$ and $D_2$. All the diagonal entries of $D_1$ are
bounded above by $a^2$, those of $D_2$ are at most $b^2$, and hence
$\|E\|_\infty\le \sqrt{a^2+b^2}$ by Lemma~\ref{l:dual-e8}. 
\end{proof}

\begin{lemma}\label{l:disjsupp}
If $C$ is a block-diagonal matrix with blocks $A$ and $B$ on the diagonal,
then $\enorm C=\max\{\enorm A,\enorm B\}$.
\end{lemma}

\begin{proof}
  The inequality $\enorm C \geq \max\{\enorm A,\enorm B\}$ is a direct
  consequence of \eqref{e:monotone}. Next we prove the reverse direction.
If $D_1$ is the dual matrix of the ellipsoid witnessing $\enorm A$ and
similarly for $D_2$ and $B$, then the block-diagonal matrix $D$ with
blocks $D_1$ and $D_2$ on the diagonal defines an ellipsoid containing
all columns of~$C$. This is easy to check using
the formula \eqref{e:duel} defining $E(D)$ and the fact
that a sum of positive definite matrices is positive definite. The
inequality $\enorm C \leq \max\{\enorm A,\enorm B\}$ then follows from
Lemma~\ref{l:dual-e8}. 
\end{proof}

\section{Relating the $\gamma_2$ norm and hereditary
  discrepancy}\label{s:devi}

Here we prove the inequalities \eqref{e:NTineq1} and \eqref{e:NTineq2}
relating $\enorm.$ and $\herdisc(.)$. We also argue that the
\eqref{e:NTineq2} is asymptotically tight. In Section~\ref{s:ints} we
will give an example on which \eqref{e:NTineq1} is asymptotically
tight as well.

\cut{We have already seen in Section~\ref{s:ints} that 
\eqref{e:NTineq1} is asymptotically tight. 
Let us first mention a simple but perhaps useful observation, which
gives a somewhat weaker result.

There are examples of set systems $\FF_1,\FF_2$ on an $n$-point
set $X$ such that $|\FF_1|,|\FF_2|=O(n)$,
$\herdisc\FF_1$ and $\herdisc \FF_2$ are
bounded by a constant (actually by $1$), and $\herdisc(\FF_1\cup\FF_2)=
\Omega(\log n)$  \cite{palvo-wedges,newman2012beck}.
Therefore, no
quantity obeying the triangle inequality (possibly up to a constant),
such as the $\gamma_2$ norm, 
can approximate $\herdisc$ with a factor better than~$\log n$.}

\subsection{The $\gamma_2$ norm is at most \boldmath$\log m$ times
herdisc}\label{s:simplified-ineq1}

%In~\cite{NT}, inequality~\eqref{e:NTineq1} was proved 
%by using  the \emph{restricted
%invertibility principle} of Bourgain and
%Tzafriri; see \cite{bour-tza,vershynin}.
%to a singular value lower bound on discrepancy via the restricted
%invertibility principle of Bourgain and
%Tzafriri~\cite{bour-tza,vershynin}. 
%Here we give a different proof
%based only on elementary linear algebra and the determinant lower bound. 

We will actually
establish the following inequalities relating the $\gamma_2$
norm to the determinant lower bound.

\begin{theorem}\label{t:detlb-E8}
  For any $m \times n$ matrix $A$ of rank $r$,
  \begin{equation*}
    \detlb A \leq \enorm A \leq O(\log r)\cdot \detlb A.
  \end{equation*}
\end{theorem}

Inequality \eqref{e:NTineq1} is an immediate consequence of
the second inequality in the theorem (and of $r \leq \min\{m,n\}$):
\[
\enorm A\le O(\log \min\{m,n\})\cdot \detlb A\le O(\log
\min\{m,n\})\herdisc A,
\]
where the last inequality uses the Lov\'asz--Spencer--Vesztergombi
bound $\herdisc A\ge\frac12\detlb A$. 

First we prepare a lemma for the proof of Theorem~\ref{t:detlb-E8};
it is similar to an argument in~\cite{almosttight}.
As a motivation, we recall the Binet--Cauchy formula:
if $A$ is a $k\times n$ matrix, $k\le n$, then
$\det AA^T=\sum_{J} (\det A_J)^2$, where the sum is over
all $k$-element subsets $J\subseteq[n]$, and $A_J$ denotes the
submatrix of $A$ consisting of the columns indexed by $J$.
Consequently, for at least one of the $J$'s we have
$(\det A_J)^2\ge {n\choose k}^{-1}\det AA^T$. The next lemma
is a weighted version of this argument, where the columns of $A$
are given nonnegative real weights.

\begin{lemma}\label{l:binetcauchy}
  Let $A$ be an $k\times n$ matrix, and let $W$ be a
nonnegative diagonal unit-trace $n\times n$ matrix.
  Then there exists a $k$-element set $J \subseteq [n]$ such that
\begin{equation*}
   |\det A_J|^{1/k} \geq \sqrt{k/e}\cdot|\det AWA^T|^{1/2k}.    
\end{equation*}
\end{lemma}

\begin{proof} Applying the Binet--Cauchy formula to the matrix $AW^{1/2}$
and slightly simplifying, we have
  \begin{equation*}
    \det AWA^T = \sum_{J}{(\det A_J)^2\prod_{j \in J}{w_{jj}}}.
  \end{equation*}
Now $\sum_J \prod_{j \in J}{w_{jj}}\le  \frac{1}{k!} \bigl(\sum_{j=1}^n{w_{jj}}\bigr)^k =  \frac{1}{k!}$, because each term of the left-hand side appears
$k!$-times on the right-hand side (and the weights $w_{jj}$ are nonnegative
and sum to~1). Therefore
  \begin{eqnarray*}
    \det AWA^T &\leq &
    \Bigl(\max_J (\det A_J)^2\Bigr)
   \sum_J\prod_{j \in J}w_{jj}\\
    &\leq &\frac{1}{k!} \max_J (\det A_J)^2. 
  \end{eqnarray*}
 So there exists a $k$-element $J$ with
  \begin{equation*}
   |\det A_J|^{1/k} \geq (k!)^{1/2k}|\det AWA^T|^{1/2k} \geq 
\sqrt{k/e}\cdot |\det AWA^T|^{1/2k},
  \end{equation*}
where the last inequality follows from the estimate 
$k! \geq (k/e)^k$.
\end{proof}

\begin{proof}[Proof of Theorem~\ref{t:detlb-E8}]
For the inequality $\detlb A \leq \enorm A$, we first observe that
if $B$ is a $k\times k$ matrix, then 
\begin{equation}\label{e:AGsing}
|\det B|^{1/k}\le \frac 1k\|B\|_*
\end{equation}
Indeed, the left-hand side is the geometric mean of the singular values
of $B$, while the right-hand side is the arithmetic mean.

Now let $B$ be a $k\times k$ submatrix of $A$ with
$\detlb A=|\det B|^{1/k}$; then 
\[
\detlb A=|\det B|^{1/k}\le \frac 1k\|B\|_*\le \enorm B\le\enorm A.
\]

For the second inequality $\enorm A = O(\log m)\cdot \detlb A$,
the idea is, roughly speaking, to compare $\det BB^T$
and  the nuclear norm of $B$ for a (rectangular) matrix $B$ whose
singular values are all nearly the same, say within a factor of $2$, since 
then  the arithmetic-geometric inequality is
nearly an equality. Obtaining a suitable $B$ and relating
$\det BB^T$ to the determinant of a square submatrix of $A$
needs some work, and  it relies on Lemma~\ref{l:binetcauchy}.

First let $P_0$ and $Q_0$ be diagonal unit-trace matrices with
$\enorm A=\|P_0^{1/2}AQ_0^{1/2}\|$ as in Theorem~\ref{t:e8dual}.
For brevity, let us write $\tilde A:=P_0^{1/2}AQ_0^{1/2}$, and
let $\sigma_1\ge\sigma_2\ge\cdots\ge \sigma_r > 0$ 
be the nonzero singular values of~$\tilde A$.

By a standard bucketing argument (see, e.g., \cite[Lemma~7]{almosttight}),
there is some $t>0$ such that if we set $K:=\{i\in[m]: t\le\sigma_i<2t\}$,
then 
\[
\sum_{i\in K}\sigma_i\ge \Omega(\tfrac1{\log r})\sum_{i=1}^m\sigma_i.
\]
Let us set $k:=|K|$.

Next, we define a suitable $k\times n$ matrix with singular values
$\sigma_i$, $i\in K$. Let $\tilde A=U\Sigma V^T$ 
be the singular-value
decomposition of $\tilde A$, with $U$ and $V$ orthogonal and 
$\Sigma$ having $\sigma_1,\ldots,\sigma_r$ on the main diagonal.

Let $\Pi_K$ be the $k\times m$ matrix corresponding to the projection on
the coordinates indexed by $K$;
that is, $\Pi_K$ has  $1$s in positions $(1,i_1), \ldots, (k, i_k)$,
where $i_1 < \ldots < i_k$ are the elements of $K$. The matrix $\Pi_K\Sigma=\Pi_KU^T \tilde A V=U_K^T\tilde AV$
has singular values $\sigma_i$, $i\in K$, and so does the matrix
$U_K^T\tilde A$, since right multiplication by the orthogonal matrix $V^T$
does not change the singular values.

This $k\times m$ matrix $U_K^T\tilde A$ 
is going to be the matrix $B$ alluded to in the 
sketch of the proof idea above. We have
\[
|\det BB^T|^{1/2k}=\Bigl(\prod_{i\in K}\sigma_i\Bigr)^{1/k}\ge
\frac{1}{2k}\sum_{i\in K}\sigma_i=\Omega\bigl(\tfrac 1{k\log r}\bigr)\enorm A.
\]

It remains to relate $\det BB^T$ to the determinant of a square submatrix
of $A$, and this is where Lemma~\ref{l:binetcauchy} is applied---actually 
applied twice, once for columns, and once for rows.
 
First we set $C:=U_K^T P_0^{1/2}A$; then $B=CQ_0^{1/2}$. Applying
Lemma~\ref{l:binetcauchy} with $C$ in the role of $A$ and $Q_0$ in the role
of $W$, we obtain a $k$-element index set $J\subseteq [n]$ such that
\[
|\det C_J|^{1/k}\ge \sqrt{k/e}\cdot |\det BB^T|^{1/2k}.
\]
Next, we set $D:= P_0^{1/2}A_J$, and we claim that $\det D^TD\ge (\det
C_J)^2$. Indeed, we have $C_J = U_K^TD$, and, since $U$ is an
orthogonal transformation, $(U^TD)^T(U^TD) = D^T D$. Then, by the
Binet--Cauchy formula,
\begin{align*}
  \det D^T D = \det (U^TD)^T (U^TD) &= \sum_{L}{(\det U_L^TD)^2} \\
  &\geq (\det U^T_KD)^2 = (\det C_J)^2.
\end{align*}

The next (and last) step is analogous. We have $D^T=A_J^TP_0^{1/2}$,
and so we apply Lemma~\ref{l:binetcauchy} with $A_J^T$ in the role of $A$
and $P_0$ in the role of $W$, obtaining a $k$-element subset $I\subseteq[m]$
with $|\det A_{I,J}|^{1/k}\ge \sqrt{k/e}\cdot|\det D^TD|^{1/2k}$ (where $A_{I,J}$
is the submatrix of $A$ with rows indexed by $I$ and columns by $J$).

Following the chain of inequalities backwards, we  have
\begin{eqnarray*}
\detlb A&\ge &
|\det A_{I,J}|^{1/k}\ge \sqrt{k/e}\cdot|\det D^TD|^{1/2k}\ge
 \sqrt{k/e}\cdot |\det C_J|^{1/k}\\
&\ge &(k/e)|\det BB^T|^{1/2k}=
\Omega\bigl(\tfrac 1{\log r}\bigr)\enorm A,
\end{eqnarray*}
and the theorem is proved.
\end{proof}

In Section~\ref{s:ints} we will see that inequality \eqref{e:NTineq1} is
asymptotically tight.  Let us now mention a simple but
perhaps useful observation, which gives a somewhat weaker result.

There are examples of set systems $\FF_1,\FF_2$ on an $n$-point
set $X$ such that $|\FF_1|,|\FF_2|=O(n)$,
$\herdisc\FF_1$ and $\herdisc \FF_2$ are
bounded by a constant (actually by $1$), and $\herdisc(\FF_1\cup\FF_2)=
\Omega(\log n)$  \cite{palvo-wedges,newman2012beck}.
Therefore, no
quantity obeying the triangle inequality (possibly up to a constant),
such as the $\gamma_2$ norm, 
can approximate $\herdisc$ with a factor better than~$\log n$.

\subsection{The hereditary discrepancy is at most \boldmath$\sqrt{\log m}$ times
$\gamma_2$}

\cut{First, for the reader's convenience, we quickly sketch the proof
of inequality \eqref{e:NTineq2} from \cite{NT}, asserting that
$\herdisc A\le O(\sqrt{\log m}\,)\cdot\enorm A$. }

In the proof of inequality~\eqref{e:NTineq2} we use a remarkable
result of Banaszczyk, which we state next.

\begin{theorem}[\cite{banasz98}]\label{t:bana}
  Let $c_1,\ldots,c_n$ be vectors in the Euclidean unit ball
  $B^m\subset\R^m$ and let $K\subseteq\R^m$ be a convex body with Gaussian measure
  \[(2\pi)^{-m/2}\int_Ke^{-\|x\|^2/2}\dd x\ge\frac 12.\] Then there is a vector
  $x=(x_1,\ldots,x_n)\in\{-1,1\}^n$ of signs such that $\sum_{j=1}^n
  x_jc_j\in D_0\cdot K$, where $D_0$ is an absolute constant. 
\end{theorem}

We prove the following theorem:

\begin{theorem}\label{t:ub}
  For any $m\times n$ matrix $A$, 
  \[
  \disc A = O(\sqrt{\log m})\cdot \gamma_2(A).
  \]
\end{theorem}

While Theorem~\ref{t:ub} at first appears weaker than
inequality~\eqref{e:NTineq2}, it in fact implies it, due to the
monotonicity of $\gamma_2$. Indeed, by \eqref{e:monotone}, we have
\[
\herdisc A = \max_{J\subseteq[n]}\disc A_J \leq O(\sqrt{\log
  m}\,)\cdot \max_{J\subseteq [n]}\gamma_2(A_J) \leq O(\sqrt{\log m})\gamma_2(A).
\]

\begin{proof}[Proof of Theorem~\ref{t:ub}]
  Let $A = B_0C_0$ be a factorization of $A$ achieving $\gamma_2(A)$,
  such that $\|C_0\|_{1\to 2} = 1$ and $\|B_0\|_{2\to\infty} =
  \gamma_2(A)$. Without loss of generality, we can assume that $B_0$
  is an $m\times m$ matrix and $C_0$ is an $m\times n$ matrix. Let
  $b_1, \ldots, b_m\in\R^m$ be the rows of $B_0$ and $c_1, \ldots, c_n
  \in \R^m$ be the columns of $C_0$. By our choice of $B_0$ and $C_0$,
  $\|b_i\|_2 \leq \gamma_2(A)$ for all $1 \leq i \leq m$, and
  $\|c_j\|_2 \leq 1$ for all $1 \leq j \leq n$. Define the convex body
  $K := \{x: \|B_0x\|_\infty \leq D\gamma_2(A)\}$ for a scalar $D$ to
  be determined later. $K$ is the intersection of the $m$ centrally
  symmetric slabs $\{x: |b_i^T x| \leq D\gamma_2(A)\}$, $i = 1,
  \ldots, m$. By \v{S}idak's lemma (see~\cite{Ball01-survey} for a
  simple proof), the Gaussian measure of $K$ is at least the product
  of the measures of the slabs, i.e.
  \[
  (2\pi)^{-m/2}\int_Ke^{-\|x\|^2/2}\dd x \geq \prod_{i =
    1}^m{\sqrt{2\pi}\int_{-\beta_i}^{\beta_i}e^{-y^2/2}\dd y},
  \]
  where $\beta_i := \frac{D\gamma_2(A)}{\|b_i\|_2} \geq D$ is the
  half-width of the $i$-th slab. By standard
  Gaussian concentration results, we have
  $\sqrt{2\pi}\int_{-\beta_i}^{\beta_i}e^{-y^2/2}\dd y \geq 1 -
  e^{-\beta_i^2/2} \geq 1 - e^{-D^2/2}$, and, therefore,
  \[
  (2\pi)^{-m/2}\int_Ke^{-\|x\|^2/2}\dd x  \geq (1-e^{-D^2/2})^m.
  \]
  Letting $D$ be a suitable constant multiple of $\sqrt{\log m}$, the
  above inequality implies that the Gaussian measure of $K$ is at
  least $1/2$. We can then apply Theorem~\ref{t:bana} and conclude
  that there exists a vector of signs $x = (x_1, \ldots, x_n) \in
  \{-1, 1\}^n$ so that
  \[
  \sum_j{x_j c_j} \in D_0\cdot K \Leftrightarrow \|Ax\|_\infty =
  \Bigl\|\sum_j{x_j Bc_j}\Bigr\|_\infty \leq D_0D\cdot\gamma_2(A).
  \]
  Since $D = O(\sqrt{\log m}\,)$, this completes the proof.
\end{proof}

An argument similar to the proof above was used by Larsen in his work
on oblivious data structures in the group model~\cite{larsen14}.

\cut{Now let $A$ be a given $m\times n$ matrix, let $E$ be an ellipsoid
witnessing $\enorm A$, and for convenience, let us assume
w.l.o.g.\ that $\enorm A=1$, i.e., that $C^m=[-1,1]^m$ is the smallest
cube containing~$E$.
Let $T\:\R^m\to\R^m$ be a linear map
such that $T(E)=B^m$, as is illustrated in the following picture:
\immfig{banapf}
To show that $\disc A\le D$ for some $D$
means finding signs $x_1,\ldots,x_n$
such that $\sum_{j=1}^n x_ja_j\in D\cdot C^m$, where $a_j$ is the
$j$th column of $A$ and $D\cdot C^m$ is the unit cube blown up $D$-times.
This in turn is equivalent to $\sum_{j=1}^n x_jT(a_j)\in D\cdot K$,
where $K=T(C^m)$. Banaszczyk's theorem guarantees the latter to be possible
provided that $\gamma(K')\ge \frac12$, where $K'=(D/D_0)K$. 
Our $K'$ is the intersection of $m$ centrally symmetric slabs, each of width 
$2D/D_0$, and  standard facts about the Gaussian measure (\v{S}id\'ak's lemma)
guarantee that $\gamma(K')\ge (1-e^{-(D/D_0)^2/2})^m$. Letting $D$ be a suitable
constant multiple of $\sqrt{\log m}$ yields $\gamma(K')\ge\frac12$
and concludes the proof.}

Next, we show that $\sqrt{\log m}$ in inequality \eqref{e:NTineq2} cannot be
replaced by any asymptotically smaller factor.

\begin{theorem}\label{t:biggerherdisc}
For all $m$, there are $m\times n$ matrices $A$, with $n=\Theta(\log m)$,
such that 
\[
\disc A\ge \Omega(\sqrt{\log m}\,)\cdot\enorm A.\
\]
\end{theorem}

\begin{proof} A very simple example is the incidence matrix $A$ of the
  system of all subsets of $[n]$, with $m=2^n$, whose discrepancy is
  $n/2=\Theta(\log m)$.  Indeed, the characteristic vectors of all
  sets have Euclidean norm at most $\sqrt n$, and hence, using the
  trivial factorization $A = AI$, the $\gamma_2$ norm is at most
  $\sqrt n=O(\sqrt{\log m}\,)$.

Here is another proof, which perhaps provides more insight into
the geometric reason behind the theorem.
Let us consider the unit cube $C^m:=[-1,1]^m$ in $\R^m$.
By the quantitative Dvoretzky theorem, there is a linear subspace 
$F\subset\R^m$ of dimension $k=\Theta(\log m)$ such that the slice
$S:=F\cap C^m$ is $2$-almost spherical; that is, if $B_F$ denotes the
largest Euclidean ball in $F$ centered at $0$ contained in $S$, then
$S\subseteq 2B_F$ (see, e.g., \cite[Lect.~2]{Ball-ln}). 
 Let $r$ be the radius of~$B_F$.

Let us choose a system $a_1,\ldots,a_k$ of orthogonal vectors in $B_F$
of length $r$. These are the columns of the
matrix~$A$.

We have $\enorm A\le 1$, since  $B_F$ is a (degenerate) ellipsoid containing
the $a_i$ and contained in~$C^m$. (We are using the geometric
definition of $\gamma_2$ here.)

Every linear combination $\sum_{i=1}^k x_ia_i$, where $x_i\in\{-1,1\}$, has Euclidean norm $r\sqrt k$, and hence it does not belong to the cube
$D\cdot C^m$ for any $D<\frac 12\sqrt k$. So $\disc A\ge \frac12\sqrt k=
\Omega(\sqrt{\log m}\,)$.
\end{proof}

\section{The $\gamma_2$ norm for intervals}\label{s:ints}

In this section we deal with a particular example: the system
$\II_n$ of all initial segments $\{1,2,\ldots,i\}$, $i=1,2,\ldots,n$,
of $\{1,2,\ldots,n\}$. Its incidence matrix is $T_n$, 
the $n\times n$ matrix with $0$s above the main
diagonal and $1$s everywhere else. 

It is well known, and easy to see,
that $\herdisc T_n=1$. We will prove that $\enorm{T_n}$ is of order 
$\log n$. This shows that the $\gamma_2$ norm can be
$\log n$ times larger than the hereditary discrepancy, and thus
the inequality \eqref{e:NTineq1} is asymptotically tight.
Moreover, this example is one of the key ingredients in the proof of the
lower bound on the $d$-dimensional Tusn\'ady
problem. 

\begin{prop}\label{p:ints}
We have $\enorm{T_n}=\Theta(\log n)$.
\end{prop}

%For the lower bound we present two proofs.
The upper bound is easy but we discuss it a little in
Section~\ref{s:TnUB}. The lower bound can be proved by combining
results from~\cite{ForsterSS01}
and~\cite{LinialMSS07-signmatrices}. Forster et al.~consider the $n\times n$
sign matrix $\tilde{T}_n$ with entries equal to $1$ above the main
diagonal and $-1$ everywhere else. They show that \cut{(in the terminology
of~\cite{LinialMSS07-signmatrices})} the margin complexity of
$\tilde{T}_n$ is $\Omega(\log n)$; since Linial et al.~proved
in~\cite{LinialMSS07-signmatrices} that the margin complexity of any
matrix is a lower bound on its $\gamma_2$ norm, it follows that
$\gamma_2(\tilde{T}_n) = \Omega(\log n)$. Using the equality $T_n = \frac12 J_n -
\frac12 T_n$, where $J_n$ is the $n$ by $n$ all-ones
matrix, and the triangle inequality for $\gamma_2$, we get $\gamma_2(T_n)
\geq \frac12 \gamma_2(\tilde{T}_n) - \frac12 = \Omega(\log n)$ as
well. Below we give a  more direct proof of the lower bound
using the dual characterization of $\gamma_2$ from
Theorem~\ref{t:e8dual}. 

\subsection{Lower bound on $\enorm{T_n}$ }\label{s:singvals}

\begin{proof}[Proof of the lower bound in Proposition~\ref{p:ints}]
The nuclear norm $\|T_n\|_*$  can be computed exactly
(we are indebted to Alan Edelman and Gil Strang for this fact);
namely, the singular values of $T_n$ are
\[
\frac1{2\sin \frac{(2j-1)\pi}{4n+2}},\ \ \ j=1,2,\ldots,n.
\]
Using the inequality $\sin x\le x$ for $x\ge 0$, we get
\[\enorm{T_n}\ge \frac 1n\|T_n\|_*\ge\frac {2n+1}{\pi n}\sum_{j=1}^n\frac 1{2j-1}=\Omega(\log n),\] as needed.

The singular values of $T_n$ can be obtained from the eigenvalues of
the matrix $S_n:=(T_nT_n^T)^{-1}$ which, as is not difficult to check,
has the following simple tridiagonal form:
\[
\begin{pmatrix}
2& -1& 0&  0&0& \ldots& 0&0&0\\
-1&2&-1&0&0&\ldots&0&0&0\\
0&-1&2&-1&0&\ldots&0&0&0\\
\vdots&\vdots&\vdots&\vdots&\vdots&\vdots&\vdots&\vdots&\vdots\\
0&0&0&0&0&\ldots&-1&2&-1\\
0&0&0&0&0&\ldots&0&-1&1
\end{pmatrix}
\]
(the $1$ in the lower right corner is exceptional; the rest
of the main diagonal are $2$s). By general properties of eigenvalues
and singular values, if $\lambda_1,\ldots,\lambda_n$ are the
eigenvalues of $S_n$, then the singular values
of $T_n$ are $\lambda_1^{-1/2},\ldots, \lambda_n^{-1/2}$.
The eigenvalues of $S_n$ are computed, as a part of more general theory, 
in Strang and MacNamara \cite[Sec.~9]{StraMacNama}; the calculation
is not hard to verify since they also give the eigenvectors
explicitly.

One can also calculate the characteristic polynomial $p_n(x)$ of $S_n$:
it satisfies the recurrence $p_{n+1} = (2-x)p_n - p_{n-1}$
with initial conditions $p_1 = 1-x$ and $p_0 = 1$,
from which one can check that $p_n(x) =
U_n\bigl(\frac{2-x}{2}\bigr) - U_{n-1}\bigl(\frac{2-x}{2}\bigr)$,
where $U_n$ is the degree-$n$ Chebyshev polynomial of the second
kind. The claimed roots of $p_n$ can then be verified using
the trigonometric representation of $U_n$.
%Alternatively, the eigenvalues can be computed as roots of the
%characteristic polynomial $p_n(x)$ of $S_n$. Observe that $p_n$
%satisfies the recurrence formula
%$p_{n+1} = (2-x)p_n - p_{n-1}$,
%with initial conditions $p_1 = 1-x$ and $p_0 = 1$. From the
%recurrence, we can verify that $p_n(x) =
%U_n\bigl(\frac{2-x}{2}\bigr) - U_{n-1}\bigl(\frac{2-x}{2}\bigr)$,
%where $U_n$ is the degree $n$ Chebyshev polynomial of the second
%kind. The equation $U_n\bigl(\frac{2-x}{2}\bigr) =
%U_{n-1}\bigl(\frac{2-x}{2}\bigr)$ is easily solved by considering the
%trigonometric representation of $U_n$.
\end{proof}

\heading{Lower bound by Fourier analysis.}
The lower bound in Proposition~\ref{p:ints} can also be proved by
relating $T_n$ to a circulant matrix, whose singular values can be
estimated using Fourier analysis. Observe that if we put four copies
of $T_n$ together in the following way
\[
\begin{pmatrix}
    T_n & T_n^T\\
    T_n^T & T_n
  \end{pmatrix},
\]
we obtain a circulant matrix, which we denote by $\conv{n+1,2n}$;
for example, for $n=3$, we have
\[
C_{4,6}=\begin{pmatrix}
    1&0&0&1&1&1\\
    1&1&0&0&1&1\\
    1&1&1&0&0&1\\
    1&1&1&1&0&0\\
    0&1&1&1&1&0\\
    0&0&1&1&1&1
  \end{pmatrix}.
\]
We have $\|\conv{n+1,2n}\|_* \leq 4\|T_n\|_*$ by the triangle inequality
for the nuclear norm (and since $\|T_n^T\|_*=
\|T_n\|_*$ and adding zero rows or columns does not change $\|.\|_*$).
Thus, it suffices to prove $\|\conv{n+1,2n}\|_*=\Omega(n\log n)$.

Let $c$ be the first column of $\conv{n+1,2n}$, i.e.~a vector of $n+1$
ones followed by $n-1$ zeros, and let us use the shorthand $C :=
\conv{n+1,2n}$.  Let further $\omega = e^{-i 2\pi/n}$, where $i = \sqrt{-1}$
is the imaginary unit.  It is well known that the eigenvalues of a
circulant matrix with first column $c$ are the Fourier coefficients
$\hat{c}_0, \ldots, \hat{c}_{n-1}$ of $c$:
\[
\hat{c}_j = \sum_{k = 0}^{s-1}{\omega^{jk}} =
\frac{\omega^{js}- 1}{\omega^{j}-1}.  
\]
Since $C$ is a normal matrix (because $C^TC = CC^T$), its singular
values are equal to the absolute values of its eigenvalues. Therefore,
$\|C\|_* = \sum_{j = 0}^{n-1}{|\hat{c}_j|}$, so we need to bound this
sum from below by $\Omega(n\log n)$. The sum can be estimated
analogously to the well-known estimate of the $L_1$ norm of the
Dirichlet kernel, giving the desired bound.

\subsection{An asymptotic upper bound and optimal ellipsoids}\label{s:TnUB}

There are several ways of showing $\enorm{T_n}=O(\log n)$.  One of
them is using $\herdisc T_n=1$ and the inequality \eqref{e:NTineq1}
relating $\gamma_2$ to $\herdisc$.  Here is another, explicit argument
using the triangle inequality. As the next picture indicates,
\immfig{trisum} the lower triangular matrix $T_n$ can be expressed as
$T_n=A_1+\cdots+A_t$, $t=O(\log n)$. (The shaded regions contain $1$s
and the white ones $0$s; the picture is for $n=8$.)  This
decomposition corresponds to the decomposition of intervals into
canonical (binary) ones, which is a standard trick in discrepancy
theory.

We have $\enorm{A_i}=1$ for each $i$: an all-ones matrix
has $\gamma_2$ norm $1$ (since it can be factored as the outer
product of the all-ones vector with itself),
and each $A_i$ can be obtained
from all-ones matrices by the block-diagonal construction 
as in Lemma~\ref{l:disjsupp}
and
by adding zero rows and columns. Hence $\enorm {T_n}\le
\sum_{i=1}^t\enorm{A_i}=O(\log n)$.

The upper bound obtained from this argument is actually
$\lfloor\log_2n\rfloor+1$. Using the semidefinite programming
formulation in~\cite{LinialMSS07-signmatrices} and the SDP solvers
SDPT3 and SeDuMi (for verification), with an interface through Matlab
and the CVX system, we have calculated the values of $\enorm{T_n}$ and
the corresponding optimal primal and dual solutions numerically, for
$n$ up to $2^7=128$.

\labfigw{optEllv2}{9.5cm}{Bounds on $\enorm{T_n}$: $\lceil\log_2
  n\rceil + 1$ (top curve), Fredman's factorization (second curve),
  the actual value computed by an SDP solver (third curve), and the
  lower bound $\frac 1n\|T_n\|_*$. The $x$-axis shows $\log_2 n$.}

%\sasho{I think we lost the actual values and only have the
%  plot. Should we recompute, maybe using the Linial et al. SDP?}

The resulting values of $\enorm{T_n}$ are shown in
Fig.~\ref{f:optEllv2}, together with the $\lceil\log_2n\rceil+1$ upper
bound and the lower bound of $\frac 1n\|T_n\|_*$ as in
Section~\ref{s:singvals}.  One can see that while $\frac 1n\|T_n\|_*$
is quite a good approximation, it is not tight, and also that the
upper bound $\lceil\log_2n\rceil+1$ overestimates the actual value
almost four times. We have also plotted the value of
$\|B\|_{2\to\infty}\|C\|_{1\to 2}$ for a factorization $T_n = BC$ due
to Fredman~\cite{Fredman82}; asymptotically, the value achieved by
his method is $\log_\lambda n + O(\log \log n)$ for $\lambda = 3 +
2\sqrt{2}$.  Fredman's factorization uses matrices with entries in
$\{-1, 0, 1\}$ and it is asymptotically optimal over such
factorizations with respect to $\max\{\|B\|_{2\to\infty}, \|C\|_{1\to
  2}\}$.

It would be interesting to find the exact value of $\enorm{T_n}$
theoretically and to understand what the optimal ellipsoids look
like. Fig.~\ref{f:surf50mth} shows a 3-dimensional plot of the entries
of the dual matrix $D$ of an optimal ellipsoid for $T_{50}$; the two
horizontal axes correspond to the rows and columns of $D$, and the
vertical axis shows the magnitude of the entries. Similarly, in
Fig.~\ref{f:plot-dual50} we have
plotted the diagonal entries of $P$ in an optimal solution $(P,Q)$ to
the dual program in Theorem~\ref{t:e8dual};
the horizontal axis corresponds to an index $i$ and the vertical axis
to the value $P_{ii}$. The diagonal entries of the optimal $Q$ appear
to be identical to those of $P$ after a rearrangement: $Q_{ii} =
P_{n-i + 1, n-i + 1}$. For both plots the values are defined for
integer indexes only, and we have used interpolation
to produce smooth graphs. It seems that, as $n\to\infty$, the
matrices of the optimal ellipsoids should converge (in a suitable
sense) to some nice function, and so should the optimal dual
solutions, but we do not yet have a guess what these functions might
be---they may very well be known in some area of mathematics.

\ifbigpic
\labfigw{surf50mth}{10cm}{The dual matrix of an optimal ellipsoid
for $T_{50}$.}
\labfigw{plot-dual50}{10cm}{Diagonal entries of an optimal dual solution
  for $T_{50}$.}
\fi

\section{General theorems about discrepancy}\label{s:disc-thms}

\heading{Union of set systems. } 
Using the inequality in Lemma~\ref{l:union} and inequalities 
\eqref{e:NTineq1},\eqref{e:NTineq2}, we obtain the following
result, which is a somewhat sharper version of a theorem proved
in \cite{almosttight} using the determinant lower bound:

\begin{theorem}[Union of set systems] \label{t:un}
Let $\FF_1,\ldots,\FF_t$ be set systems on
an $n$-point ground set $V$, and let
$\FF=\FF_1\cup\cdots\cup\FF_t$.
Then
\[
\herdisc\FF \le O\Bigl(\sqrt{\log|\FF|}\Bigr)
\biggl(\sum_{i=1}^t(\log|\FF_i|)^2(\herdisc \FF_i)^2\biggr)^{1/2}.
\]
\end{theorem}

We note that if the set systems $\FF_1$ and $\FF_2$ have disjoint
ground sets, then $\herdisc(\FF_1\cup\FF_2)=\max(\herdisc\FF_1,
\herdisc\FF_2)$, which can be regarded as a counterpart of
Lemma~\ref{l:disjsupp}.

\heading{Building sets from disjoint pieces. }
In a similar vein, the triangle inequality for $\gamma_2$
together with \eqref{e:NTineq1},\eqref{e:NTineq2} immediately yield the
following consequence:

\begin{theorem}
\label{t:disjpieces} Let $\FF_1,\ldots,\FF_t$ be set systems on
an $n$-point ground set $V$, and let
$\FF$ be a set system such that for each $F\in\FF$ there
are pairwise disjoint sets $F_1\in\FF_1$,\ldots, $F_t\in\FF_t$
so that $F=F_1\cup\cdots\cup F_t$. Then
\[
\herdisc\FF \le 
O\Bigl(\sqrt{\log|\FF|}\Bigr)\sum_{i=1}^t(\log|\FF_i|)\herdisc \FF_i.
\]
\end{theorem}

In Section~\ref{s:tusnady} below, we will obtain an example showing
that if each of the systems $\FF_i$ in the theorem has hereditary discrepancy
at most $D$, the system $\FF$ may have discrepancy about $tD$, up
to a logarithmic factor, and thus in this sense, the theorem
is not far from worst-case optimal.

%\heading{A counterexample for a possible stronger inequality. }
%Theorem~\ref{t:disjpieces} shows that if $\FF_1,\ldots,\FF_t$
%are set systems of hereditary discrepancy at most $D$,
%and each set of a system $\FF$ is of the form $F_1\cup\cdots\cup F_t$,
%with $F_1\in\FF_1,\ldots,F_t\in\FF_t$ pairwise disjoint,
%then $\FF$ has hereditary discrepancy at most about $tD$,
%up to polylogarithmic factors. However, in some examples,
%$\herdisc\FF$ is actually only about $\sqrt t\cdot  D$,
%again up to polylogarithmic factors.
%

\heading{Product set systems. } Let $\FF$ be a set system on a ground
set $V$, and $\GG$ a set system on a ground set $W$. Following
Doerr, Srivastav, and Wehr \cite{Doerr-al-productdisc} (and probably
many other sources), we define the product
$\FF\times\GG$ as the set system $\{F\times G:F\in\FF, G\in\GG\}$
on $V\times W$.

Since the incidence matrix of $\FF\times\GG$ is the Kronecker product
of the incidence matrices of $\FF$ and $\GG$, from Theorem~\ref{t:krone}
and the usual inequalities \eqref{e:NTineq1},\eqref{e:NTineq2},
we get that the hereditary discrepancy is approximately multiplicative:

\begin{theorem}\label{t:mult}
Let $\FF_1,\ldots,\FF_t$ be set systems, let $m_i=|\FF_i|>1$ for all $i$,
let $\FF=\FF_1\times\cdots\times\FF_t$,
and let $D:=\prod_{i=1}^t\herdisc\FF_i$. Then
\[
\frac{D}
{C^t \log |\FF| \prod_{i=1}^t\sqrt{\log m_i}}
\le \herdisc\FF \le 
D\cdot C^t\sqrt{\log |\FF|}\prod_{i=1}^t\log m_i
\]
with a suitable absolute constant~$C$.
\end{theorem}

In the proof of the bounds for Tusn\'ady's problem in Section~\ref{s:tusnady}
we will see that the upper bound is not far from being tight.
Here we give a simple example showing that the lower bound is
near-tight as well.

Let $m=2^k$ with $k$ even and let $\PP=2^{[k]}$ be the system of all subsets
of the $k$-element set $[k]$. Then $|\PP|=m$ and $\herdisc\PP=k/2$.
The lower bound in the theorem for the hereditary discrepancy of the 
$t$-fold product $\PP^t$, assuming $t$ constant, is of order
$D/\log^{t/2+1}m$, where $D=\herdisc(\PP)^t$. On the other hand,
it is well known that any system of $M$ sets on $n$ points has
discrepancy $O(\sqrt{n\log M}\,)$ (this is witnessed by a random coloring;
see, e.g., \cite{s-tlpm-87,Chazelle-book,m-gd}), which in our case,
with $n=k^t$ and $M=m^t$, shows that $\herdisc(\PP^t)$
is at most of order $k^{t/2+1/2}\approx D/(\log m)^{t/2-1/2}$,
which differs from the lower bound only by \cut{the (usual?)} a
factor of $\log^{3/2}m$, independent of~$t$.

\section{On Tusn\'ady's problem}
\label{s:tusnady}

\begin{proof}[Proof of Theorem~\ref{t:tusnady}] The proof
was already sketched in the introduction, so here we
just present it slightly more formally.
Let $\AA_d\subseteq\RR_d$ be the set of all \emph{anchored}
axis-parallel boxes, of the form $[0,b_1]\times\cdots\times [0,b_d]$.
Clearly $\disc(n,\AA_d)\le\disc(n,\RR_d)$, and since every box
$R\in\RR_d$ can be expressed as a signed combination of at most $2^d$
anchored boxes, we have $\disc(n,\RR_d)\le 2^d\disc(n,\AA_d)$.

Let us consider the $d$-dimensional
grid $[n]^d\subset\R^d$ (with $n^d$ points), and let
$\GG_{d,n}=\AA_d([n]^d)$ be the subsets induced on it by anchored boxes.
It suffices to prove that $\herdisc \GG_{d,n}=\Omega(\log^{d-1}n)$,
and for this, in view of inequality \eqref{e:NTineq1},
it is enough to show that $\enorm{\GG_{d,n}}=\Omega(\log^d n)$.

Now $\GG_{d,n}$ is (isomorphic to) the $d$-fold product $\II_n^d$
of the system of initial segments in $\{1,2,\ldots,n\}$,
and so $\enorm{\GG_{d,n}}=\enorm{T_n}^d=\Theta(\log^dn)$
(Theorem~\ref{t:krone} and Proposition~\ref{p:ints}).

This finishes the proof of the lower bound. To prove the upper bound
$\disc(n,\RR_d)=O(\log^{d+1/2}n)$, we consider an arbitrary
$n$-point set $P\subset\R^d$. Since the set system $\AA_d(P)$
is not changed by a monotone transformation of each of the coordinates,
we may assume $P\subseteq [n]^d$. Hence
\[
\disc(\AA_d(P))\le\herdisc \GG_{d,n}
\le O(\enorm{\GG_{d,n}}\sqrt{\log n^d}\,)=
O(\log^{d+1/2}n).
\]
\end{proof}

\heading{Near-optimality of the bounds in Theorems~\ref{t:disjpieces}
and~\ref{t:mult}.}
In  Theorem~\ref{t:mult} (discrepancy for the product
of set systems), if we set $\FF_i=\II_m$ for all $i=1,2,\ldots,t$,
then the product of $\herdisc \FF_i$ is $D=1$, while 
the hereditary discrepancy of the product is $\Omega(\log^{t-1}m)$
assuming $t$ constant. The upper bound in Theorem~\ref{t:mult}
is $O(\log^{t+1/2}m)$.

For Theorem~\ref{t:disjpieces} (sets made of disjoint pieces),
we take $\FF$ to be the set system $\GG_{d,n}$ induced on the
grid $[n]^d$ by anchored axis-parallel boxes, with hereditary discrepancy
at least $\Omega(\log^{d-1}n)$. 

To define the systems $\FF_i$, we use \emph{canonical binary boxes}.
First let us define a \emph{(binary) canonical interval} in $[n]$
as a set  of the form $I=[a2^i,(a+1)2^i)\cap [n]$,
with $i$ and $a$ nonnegative integers.
Let us call $2^i$ the \emph{size} of such a canonical interval.
As is well known, and easy to see, every initial interval
$J=\{1,2,\ldots,j\}\subseteq[n]$ can be expressed as a disjoint
union of canonical intervals, with at most one canonical intervals
for every size (and consequently, there are $O(\log n)$
canonical intervals in the union).

Next, a \emph{canonical box} in $[n]^d$ is a product 
$B=I_1\times\cdots\times I_d$
of canonical intervals. The
\emph{size} of $B$ is the $d$-tuple $(2^{i_1},\ldots,2^{i_d})$, where
$2^{i_j}$ is the size of $I_j$. Clearly, every 
set in $\GG_{d,n}$  is a disjoint union of $O(\log^d n)$
canonical boxes, at most one for every possible size.

Let $T$ be the set of all sizes of canonical boxes, $|T|=\Theta(\log^d n)$,
and for every  size $s\in T$, let $\BB_s$ be the
system of all canonical boxes of size $s$, plus the empty set. By the above,
each set of $\GG_{d,n}$ is a disjoint union $\bigcup_{s\in T} B_s$
for some $B_s\in\BB_s$, and so the $\BB_s$ can play the role
of the $\FF_i$ in Theorem~\ref{t:disjpieces}, with $t=|T|$.

We have $D=\herdisc\BB_s=1$ for every $s$ (since the canonical boxes
of a given size are pairwise disjoint), and so 
$\herdisc\GG_{d,n}=\Omega(t^{1-1/d})D$ for every constant $d$.
Hence if, in the setting of Theorem~\ref{t:disjpieces},
 $D=\max\herdisc\FF_i$, we cannot bound $\herdisc\FF$ by
$O(t^{1-\delta}(\log |\FF|)^c D)$ for any fixed $c$ and $\delta>0$ (unlike for 
the union of set systems in Theorem~\ref{t:un}, 
where the bound is roughly $\sqrt t\cdot D$,
up to a logarithmic factor).

%end of cut

\subsection{Discrepancy of boxes in high dimension} \label{s:hibox}

Chazelle and Lvov~\cite{chl-tbhd,chaz-lvov-boxes} investigated
the hereditary discrepancy of the set system $\CC_d:=\RR_d(\{0,1\}^d)$, 
the set system induced by
axis-parallel boxes on the $d$-dimensional Boolean cube
 $\{0,1\}^d$. In other words, the sets in $\CC_d$
are subcubes of $\{0,1\}^d$.
Unlike for Tusn\'ady's problem where $d$
was considered fixed, here one is interested in the asymptotic behavior 
as $d\to\infty$.

Chazelle and Lvov proved $\herdisc \CC_d=\Omega(2^{cd})$ for an 
absolute constant $c \approx 0.0477$, which was later improved
to $c=0.0625$ in \cite{NT-HAP} (in relation to the hereditary discrepancy of
homogeneous arithmetic progressions). Here we obtain an optimal value
of the constant~$c$:

%$\herdisc \CC_d = 2^{d/16}$ was proved. Here, we use
%Theorem~\ref{t:krone} to determine the correct constant in the
%exponent of the hereditary discrepancy. The proof reduces to a simple
%computation in two dimensions.

\begin{theorem}\label{t:boxes-highdim}
  The system $\CC_d$ of subcubes of the $d$-dimensional Boolean cube 
satisfies
  \[
  \herdisc \CC_d = 2^{c_0 d + o(d)},
  \]
  where $c_0 = \log_2(2/\sqrt 3)\approx 0.2075$. The same bound holds
for the system $\AA_d(\{0,1\}^d)$ of all subsets of the cube 
induced by anchored
boxes.
\end{theorem}

%\begin{corol}
%  For any $d = \Omega(\log n)$, 
%  $\disc(n,\RR_d) = \Omega(n^{c_0})$. 
%\end{corol}

\begin{proof} The number of sets in $\CC_d$ is $3^d$, and so
in view of inequalities
\eqref{e:NTineq1} and \eqref{e:NTineq2} it suffices  to prove
$\enorm{\CC_d} = \enorm{\AA_d(\{0,1\}^d)}=2^{c_0 d}$. 

The system $\CC_d$ is the $d$-fold product $\CC_1^d$, and so
by Theorem~\ref{t:krone}, $\enorm{\CC_d}=\enorm{\CC_1}^d$.
The incidence matrix of $\CC_1$ is
\[
A =
\begin{pmatrix}
  1 & 1\\
  1 & 0\\
  0 & 1
\end{pmatrix}.
\]
To get an upper bound on $\enorm{A}$, we exhibit an appropriate
ellipsoid; it is more convenient to do it for $A^T$, since this
is a planar problem. The optimal ellipse containing
the rows of $A$ is  $\{x\in\R^2: x_1^2 +x_2^2 -
x_1x_2 \leq 1\}$; here are a picture and the dual matrix:
\[
\raisebox{-1.5cm}{\includegraphics{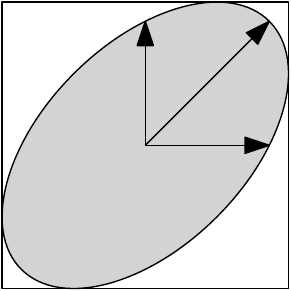}}\ \ \ \ \ \ \ \ \ \ \ \  D = \begin{pmatrix}
  \frac43 &\frac 13\\[1mm]
  \frac13 &\frac43
\end{pmatrix}.
\]
Hence $\enorm{A}\le 2/\sqrt 3$. The same ellipse also works for
the incidence matrix of the system $\AA_1(\{0,1\})$, which is the familiar
lower triangular matrix~$T_2$.

There are several ways of bounding $\enorm {T_2}\le\enorm A$
from below. For example, we can use Theorem~\ref{t:e8dual}
with 
\[
P=\begin{pmatrix}\frac 13&0\\[1mm]0&\frac 23\end{pmatrix},\ \ \ 
Q=\begin{pmatrix}\frac 23&0\\[1mm]0&\frac 13\end{pmatrix}.
\]
With some effort (or a computer algebra system) one can check
that the singular values of $P^{1/2}T_2Q^{1/2}$ are
$\frac1{\sqrt 3}\pm \frac 13$, and hence the nuclear norm
is $2/\sqrt 3$ as needed.

Alternatively, one can also check the optimality of the ellipse above
by elementary geometry, or exhibit an optimal solution of the dual
semidefinite program for $\enorm{T_2}$.
%
%
%To show that $\enorm{A} \geq
%\frac{2}{\sqrt{3}}$, we use the dual characterization in
%Theorem~\ref{t:e8dual}. Let us define $P$ to be the diagonal matrix
%with entries $(\frac23,\frac16, \frac16)$ and $Q := \frac12 I_2$. Then
%$P^{1/2}AQ^{1/2}$ has  nonzero singular values $\frac{\sqrt{3}}{2}$
%and $\frac{\sqrt{3}}{6}$,  adding up to $\frac{2\sqrt{3}}{3} = \frac{2}{\sqrt{3}}$. 
%
%We note that $\enorm{T_2} = \frac{2}{\sqrt{3}}$, and, therefore,
%the bounds in Theorem~\ref{t:boxes-highdim} hold for the system
%induces by anchored boxes on $\{0,1\}^d$ as well.
\end{proof}

\section{On combinatorial \boldmath$L_p$-discrepancy}
\label{s:Lpbound}

\heading{$L_p$-discrepancy in the continuous setting.}
Roth's beautiful argument \cite{r-id-54} for the lower bound 
$D(n,\RR_d)=\Omega(\log^{(d-1)/2}n)$ actually bounds
the discrepancy of an \emph{average} anchored axis-parallel box.
More precisely, Roth introduced the $p=2$ case of the following
notion of \emph{$L_p$-discrepancy} of an $n$-point set $P
\subset [0,1]^d$ with respect to anchored boxes, defined by
\[
D_p(P,\AA_d) := \left(\int_{[0,1]^d}{\Bigl||P\cap A(x)|-\lambda^d(A(x))\Bigr|^p dx}\right)^{1/p},
\]
where $A(x) := [0, x_1] \times \ldots \times [0,x_d]$. 
This kind of discrepancy has also been investigated extensively since
then, and its importance, e.g.~for the theory of numerical integration,
is comparable to the original ``worst--case'' discrepancy $D(P,\RR_d)$.

While the asymptotic behavior of $D(n,\RR_d)$ remains a mystery,
it turns out that the $L_p$-discrepancy $D_p(n,\AA_d)$
is of order $\log^{(d-1)/2}n$ for every fixed $p$ and $d$,
matching Roth's lower bound.
This was shown by Davenport \cite{d-nid-56} for $d=p=2$,
by Roth \cite{r-idIV-80} for $p=2$ and all $d$, and by
Chen \cite{c-id-81} (for all $p$).

\heading{Combinatorial $L_p$-discrepancy. }
A similar kind of average discrepancy can also be considered in the
combinatorial setting, as was done, e.g., in \cite{s-idbsm-96,m-beckfilp}.
Namely, for a set system $\FF$ on the ground set $[n]$ we set
\[
\disc_p\FF:=\min_{x \in \{-1, 1\}^n}\biggl(
\frac1{|\FF|}\sum_{F\in\FF}|x(F)|^p\biggr)^{1/p},
\]
with $x(F)=\sum_{j\in F} x_j$.
More generally, for a nonnegative weight function $w\:\FF\to[0,\infty)$,
not identically $0$, we similarly
define 
\[
\disc_{p,w} \FF := \min_{x \in \{-1, 1\}^n}\biggl( \tfrac1{w(\FF)}
\sum_{F\in\FF}w(F)|x(F)|^p\biggr)^{1/p}.
\]

In this section we provide some general results concerning the
combinatorial $L_p$-discrepancy, and we establish a lower
bound for anchored boxes (an $L_p$-version of Tusn\'ady's problem).

For a point set $P\subset[0,1]^d$, we let $\AA_d(P)$ be the system
of all intersections of $P$ with anchored boxes as before,
and let $\overline w=\overline w_P\: \AA_d(P)\to [0,1]$ be the
weight function given by
$\overline w(F):=\lambda^d\{x \in [0,1]^d:A(x)\cap P=F\}$; that is, the weight
of a subset of $P$ is the Lebesgue measure of the set of all corners $x$
whose corresponding anchored boxes $A(x)$ intersect $P$ in~$F$.

\begin{theorem}\label{t:l2-disc-boxes}
For every fixed $d\ge 2$ and infinitely many values of $n$,
there is an $n$-point set $P\subset\R^d$ such that
\[
\disc_{2,\overline w}\AA_d(P)=\Omega(\log^{d-1}n ).
\]
\end{theorem}

Thus, the combinatorial $L_2$-discrepancy for axis-parallel anchored
boxes has the same lower bound as the worst-case discrepancy, 
and it is roughly the square of the $L_2$-discrepancy in the
Lebesgue-measure case. (Admittedly, the analogy between the
$L_2$-discrepancy in the Lebesgue-measure and combinatorial cases
is far from perfect.)

We start working towards the proof of the theorem.
First we extend the definition of $L_p$-discrepancy to matrices
in a natural way:
for an $m\times n$ matrix $A$ we set
\[
\disc_p A:=\min_{x\in\{-1,1\}^n}m^{-1/p}\|Ax\|_p.
\]
The hereditary analog, $\herdisc_p A$, is naturally defined
as $\max_{J\subseteq [n]}\disc_p A_J$.

Now let us consider a weight function $w\:[m]\to[0,\infty)$ on the rows of
$A$. It is useful to observe that the corresponding
weighted $L_p$-discrepancy of $A$ can be written using the unweighted
discrepancy of $A$ suitably modified---namely, the $i$th row
needs to be multiplied by $w(i)^{1/p}$, assuming $w$ normalized so
that $\sum_{i=1}^m w(i)=m$. Then, with this normalization of $w$
and with $W:=\diag(w)$ being the $m\times m$ matrix with the $w(i)$
on the diagonal, we can write
\[
\disc_{p,w} A=\disc_p W^{1/p}A.
\]

Let us consider the following $L_2$-version of the determinant lower bound:
\[
\detlb_2 A := \max_{J:\emptyset\ne J\subseteq[n]}
 \sqrt{|J|/m}\cdot |\det A_J^T A_J|^{1/2|J|}.
\]
The following is proved 
in the journal version of~\cite{NTZ-diffprivacy} by an easy modification
of the argument of Lov\'asz et al.~\cite{lsv-dssm-86}:

\begin{lemma}\label{l:l2-detlb}
There exists a constant $c > 0$ such that for every $m\times n$ matrix $A$,
\[
  \herdisc_2 \geq c\detlb_2 A.
\]
\end{lemma}

We use this lemma, together with a modification of our proof of
inequality \eqref{e:NTineq1}, to establish the following:

\begin{lemma}\label{l:ell2-disc-lb}
Let $A$ be an $m\times n$ matrix, let $w\:[m]\to[0,\infty)$
be a nonnegative weight function on the rows normalized so that
$\sum_{i=1}^m w(i)=m$, and let $P=\frac 1m \diag(w)$.
Then for every nonnegative diagonal matrix $Q$ with unit trace we have
\[%\begin{equation}
  %\label{e:ell2-disc-lb}
  \|P^{1/2}AQ^{1/2}\|_* = O(\log m) \herdisc_{2,w} A.
\]%\end{equation}
\end{lemma}

\begin{proof}
In the proof of Theorem~\ref{t:detlb-E8}, we showed that if $Q$ is a
non-negative diagonal unit-trace matrix, then there
exists a submatrix $D = P^{1/2}A_J$ of $P^{1/2}A$ such that 
\[
|\det D^T D|^{1/2k} = \Omega\left(\frac{1}{\sqrt{k}\log
    m}\right)\|P^{1/2}AQ^{1/2}\|_*,
\]
where $k:=|J|$.
Setting $W := mP$ and $\tilde{A} := W^{1/2}A$,
the matrix $\sqrt m\cdot D$ is a witness for 
$\detlb_2 \tilde A=\Omega(1/\log m)\cdot
\|P^{1/2}AQ^{1/2}\|_*$. The lemma then follows
by applying Lemma~\ref{l:l2-detlb} to the matrix~$\tilde A$.
\end{proof}

\begin{proof}[Proof of Theorem~\ref{t:l2-disc-boxes}]
In the proof of Theorem~\ref{t:tusnady} we have shown that
$\enorm{\GG_{d,n}}=\Omega(\log^d n)$, where $\GG_{d,n}=\AA_d([n]^d)$
is the set system induced by anchored boxes on the grid $[n]^d$.
Unwrapping the proof shows that the diagonal matrices $P$ and $Q$ witnessing
the lower bound on $\enorm{\GG_{d,n}}$ via Theorem~\ref{t:e8dual}
can actually be taken uniform,
i.e., $P=Q=\frac1NI_N$, $N=n^d$. 

Therefore,
applying Lemma~\ref{l:ell2-disc-lb} with $A$ the incidence matrix of 
$\GG_{d,n}$, $P=\frac1NI_N$, and $w\equiv 1$
the uniform weight function, we obtain $\herdisc_2 A=\Omega(\log^{d-1}n)$.
The theorem then follows from the definition of $\herdisc_2 A$
(one can check that the weights of the subsets are given by $\overline w$
as in the theorem after appropriately scaling and shifting the grid $[n]^d$).  
\end{proof}

%Together with the dual characterization of the ellipsoid
%infinity norm (Theorem~\ref{t:duallb}), 
%the lemma implies that for every set system $\FF$, there exists
%a weight function $w$ such that $\enorm{\FF} = O(\log m) \herdisc_{2,{w}}
%\FF$. Indeed, it suffices to let $P_0$ and $Q_0$
%be such that $\|P_0^{1/2}AQ_0^{1/2}\|_* = \enorm{A}$, where $A$ is
%the incidence matrix of $\FF$, and use the lemma with
%$w$ the diagonal of $mP_0$ and $Q=Q_0$.

\section{Simple proofs of known discrepancy bounds}\label{s:sbounds}

The properties of the $\gamma_2$ norm allow for surprisingly
easy proofs of some known bounds in discrepancy theory; 
we have already seen this in the case of the upper bound for
Tusn\'ady's problem. Here we add some more examples, where
we obtain slightly suboptimal results.
%Similar arguments
%mostly work via the determinant lower bound as well, using the results
%of~\cite{almosttight}.

For convenience, we first summarize the required properties.

\begin{enumerate}
\item[(A)] (Herdisc and $\gamma_2$) 
$\frac {\enorm A} {O(\log m)}  
\le \herdisc A \le \enorm A \cdot O(\sqrt{\log m})$; these are
 inequalities \eqref{e:NTineq1}, \eqref{e:NTineq2}.
\item[(B)] (Degree bound) If each point in a set system $\FF$ is in at most
$t$ sets, then $\enorm{\FF}\le\sqrt t$. (This is because
the columns of the incidence matrix $A$ are contained in the
ball of radius $\sqrt t$, or, equivalently, by the trivial
factorization $A = IA$.)
\item[(B$'$)] (Size bound) If all sets of $\FF$ have size at most $t$,
then $\enorm{\FF}\le\sqrt t$. (This is (B) and Lemma~\ref{l:transp},
or, equivalently, by the factorization $A = AI$.)
\item[(C)] (Union) If $\FF=\FF_1\cup\cdots\cup\FF_t$,
then $\enorm\FF\le \bigl(\sum_{i=1}^t\enorm{\FF_i}^2\bigr)^{1/2}$ 
(Lemma~\ref{l:union}).
\item[(D)] (Disjoint supports) If set systems $\FF_1$ and $\FF_2$
have disjoint ground sets, then $\enorm{\FF_1\cup\FF_2}=\max(\enorm{\FF_1},
\enorm{\FF_2})$ (Lemma~\ref{l:disjsupp}).
\item[(E)] (Sets from disjoint pieces)
If every set  $F\in \FF$ can be written as
a disjoint union $F_1\cup\cdots\cup F_t$, $F_i\in\FF_i$,
then $\enorm\FF\le \sum_{i=1}^t\enorm{\FF_i}$.
\item[(F)] (Product) $\enorm{\FF_1\times\FF_2}=
\enorm{\FF_1}\times\enorm{\FF_2}$.
\end{enumerate}

\heading{A bound in terms of the maximum degree. }
If $\FF$ has maximum degree $t$, i.e., no point is in more than
$t$ sets, then we get $\disc\FF=O(\sqrt{t\log m})$ by (A) and (B),
which recovers the current best bound
for this problem, due to Banaszczyk
\cite{banasz98}. However, this example is
not quite fair, since  inequality \eqref{e:NTineq2}
used in (A) relies on a more general form of Banaszczyk's estimate.

\heading{The $k$-permutation problem. }
Given a permutation $\pi$ of $\{1,2,\ldots,n\}$,
we consider the system $\PP_\pi$ of all initial segments along
$\pi$, i.e., the sets $\{\pi(1),\ldots,\pi(i)\}$,
$i=1,2,\ldots,n$. The \emph{$k$-permutation problem}
asks for the maximum discrepancy of $\PP:=\PP_{\pi_1}\cup\cdots
\cup\PP_{\pi_k}$, where $\pi_1,\ldots,\pi_k$ are $k$ permutations
of $\{1,2,\ldots,n\}$. For $k\ge 3$, the best known upper bound
is  $O(\sqrt k\cdot\log n)$ \cite{s-idbsm-96}, and it is sharp
for $k\ge 3$ fixed~\cite{newman2012beck}.

As is well known,  $\herdisc\PP_\pi\le 1$ for every $\pi$,
and so (A) and (C)  give $\disc\PP=O(\sqrt k\cdot\log^{3/2}n)$.

\heading{Arithmetic progressions. }
Let $\AP$ be the system of all \emph{arithmetic progressions} on the set
$\{1,2,\ldots,n\}$. The discrepancy of $\AP$ was considered 
in a classical paper of  Roth \cite{KFR}, who proved an
$\Omega(n^{1/4})$ lower bound. A matching upper bound of
$O(n^{1/4})$ was obtained in \cite{ms-dap-96},
after previous weaker results by several authors.

First we present a quick way of obtaining the slightly worse upper bound
of $O(n^{1/4}\log n)$.
Let $\AP_d\subseteq\AP$ consist of all arithmetic
progressions with difference exactly $d$.
Obviously $\herdisc\AP_d\le 1$ for all $d$, and so
$\enorm{\AP_d}=O(\log n)$ by (A).  
Let us set $\AP_{\ge s}:=\bigcup_{d\ge s}\AP_d$.
Since all sets in $\AP_{\ge s}$ have size at most $n/s$, we 
have $\enorm{\AP_{\ge s}}=O(\sqrt{n/s}\,)$ by (B$'$). So, for every $s$,
$\enorm{\AP}^2\le \sum_{d=1}^{s-1}\enorm{\AP_d}^2+
\enorm{\AP_{\ge s}}^2=O(s\log^2 n+n/s)$ according to (C).
 Minimizing with $s:=\sqrt n/\log n$ gives $\enorm\AP=O(n^{1/4}\sqrt{\log n})$,
and thus $\disc\AP=O(n^{1/4}\log n)$ by~(A).

A more careful analysis, combining the ideas above with
the canonical intervals trick, shows an asymptotically
optimal bound for $\enorm{\AP}$, which in turn implies
 $\herdisc\AP =O(n^{1/4}\sqrt{\log n}\,)$, a better but still
suboptimal bound.

\begin{prop}\label{p:AP}
  $\enorm{\AP} =\Theta(n^{1/4})$.
\end{prop}

\begin{proof}
  The lower bound $\enorm{\AP} = \Omega(n^{1/4})$ is implied by the
 Lov\'asz' proof of Roth's $1/4$-theorem  using eigenvalues;
  see \cite{bs-dt-95} or \cite[Sec.~1.5]{Chazelle-book}. That proof provides
 a square matrix $\tilde{A}$ 
in which each row is the sum of the indicator vectors of at
  most two disjoint arithmetic progressions in $[n]$, and such that
the smallest singular value $\sigma_{\min}$
of $\tilde{A}$ is of order $\Omega(n^{1/4})$.

  By the triangle inequality (and since removing rows does not increase
$\gamma_2$), we have $\enorm{\AP} \geq \frac12 \enorm{\tilde{A}}$.
Then
$\enorm{\tilde A}\ge \frac1n \|\tilde{A}\|_* \geq \sigma_{\min}=
\Omega(n^{1/4})$, which proves the lower bound.

Next, we do the upper bound.
% Given an interval $I = [a, b)$, let
%  us call an arithmetic progression restricted to $I$ \emph{maximal}
%  if it is not a proper subset of any other arithmetic progression
%  on $I$. 
For an interval $I\subseteq[n]$, let $\MM_I$ be the set of
  all inclusion-maximal arithmetic progressions in $I$.  We claim that
  \begin{equation}
    \label{e:maxAP-bound}
    \enorm{\MM_I}\leq \sqrt{2}|I|^{1/4}.
  \end{equation}
%  where $|I| := |b-a|$ is the \emph{size} of the interval $I =
%  [a,b)$. 

Before proving \eqref{e:maxAP-bound}, let us see why it
  implies $\enorm{\AP} = O(n^{1/4})$. We
 recall that a binary canonical interval of size
  $2^i$ is an interval of the form $I = [a2^i,(a+1)2^i)\cap [n]$,
  where $a$ and $i$ are natural numbers. Let $\MM_i$ be the union of
  the set systems $\MM_I$ over all canonical intervals $I$ of size
  $2^i$. Since $\MM_i$ is a union of set systems with 
disjoint supports, by (D) and~\eqref{e:maxAP-bound},
 $\enorm{\MM_i} \leq 2^{\frac{i}{4} + \frac12}$.

Every arithmetic progression in $[n]$
can be written as the disjoint union of arithmetic
  progressions from $\MM_0, \ldots, \MM_{k}$, $k = \lfloor \log_2 n
  \rfloor$, so that at most two maximal arithmetic progressions from
  each $\MM_i$ are taken. Property (E) then gives
%  \[
 $ \enorm{\AP} \leq \sum_{i = 0}^k{2\cdot 2^{\frac{i}{4} + \frac12}} =
  O(n^{1/4})$. 
%  \]

It remains to prove \eqref{e:maxAP-bound}. Let
  us split $\MM_I$ as $\MM_I'\cup\MM''_I$,
 where the arithmetic progressions in $\MM_I'$ have
  difference at most $|I|^{1/2}$, and those in $\MM_I''$
  have difference larger than $|I|^{1/2}$. 

Given a difference $d$, each  $c \in I$ belongs to exactly one maximal 
arithmetic progression with difference $d$, 
because such an arithmetic progression is entirely
 determined by the congruence class of $c \bmod d$. Therefore, each
  integer in $I$ belongs to at most $|I|^{1/2}$ arithmetic
  progressions in $\MM_I'$,  and, by (B), $\enorm{\MM_I'} \leq |I|^{1/4}$. 

On  the other hand, every arithmetic progression in $\MM_I''$ has size
 at most $|I|^{1/2}$, and so, by (B$'$), $\enorm{\MM_I''} \leq |I|^{1/4}$
  as well. Since $\MM_I = \MM_I' \cup \MM_I''$, we have $\enorm{\MM_I}
  \leq \sqrt{2}|I|^{1/4}$ as desired.
\end{proof}

\heading{Multidimensional arithmetic progressions.}  Doerr, Srivastav, and
Wehr \cite{Doerr-al-productdisc} considered the discrepancy of the
system $\AP^d$ of \emph{$d$-dimensional
arithmetic progressions} in $[n]^d$, which are $d$-fold Cartesian products of
arithmetic progressions. They showed that $\disc \AP^d=\Theta(n^{d/4})$. 

Their upper bound was done by a simple product coloring argument,
which does not apply to hereditary discrepancy (since the restriction
of $\AP^d$ to a subset of $[n]^d$ no longer has the structure of
multidimensional arithmetic progressions). By
Proposition~\ref{p:AP} and (F) we have
$\enorm{\AP^d} = \Theta(n^{d/4})$, and we thus obtain the
(probably suboptimal) upper 
bound $\herdisc \AP^d=O(n^{d/4}\sqrt{\log n}\,)$.

The lower bound in \cite{Doerr-al-productdisc} uses a nontrivial
Fourier-analytic argument. Here we observe that it also follows
from Lov\'asz' lower bound proof for $\disc\AP$ mentioned above,
and a product argument. Indeed,  the $d$-fold Kronecker 
product $\tilde{A}^{\otimes d}$ of the matrix
$\tilde{A}$ as in the proof of Proposition~\ref{p:AP}
has the smallest singular value $\sigma_{\min}^d = \Omega(n^{d/4})$
for every fixed $d$, and each
of its rows is the indicator vector of the disjoint union of at most
$2^d$ sets of $\AP^d$. So $\disc \AP^d
\geq 2^{-d} \disc \tilde{A}^{\otimes d} = \Omega(n^{d/4})$,
where the final equality is by the well-known fact that the
smallest singular value is a lower bound on the discrepancy of 
a square matrix (see  \cite[Sec.~4.2]{m-gd} or
\cite[Sec.~1.5]{Chazelle-book}).

%\section{Questions}
%
%Linear-algebraic algo for Tusnady colorings?

\subsection*{Acknowledgments}

We would like to thank Alan Edelman and Gil Strang for invaluable
advice concerning the singular values of the matrix in
Proposition~\ref{p:ints}, and Van Vu for recommending the right
experts for this question. We would also like to thank Noga Alon and
Assaf Naor for pointing out that the geometric quantity
in~\cite{NT,e8-tusnady} is equivalent to the $\gamma_2$ norm. We also
thank Imre B\'ar\'any and Vojt\v{e}ch T\accent23uma for useful
discussions.

\bibliographystyle{alpha}
\bibliography{cg}

\begin{thebibliography}{LMSS07}

\bibitem[ABC97]{abc-gd-97}
J{.\,R.} Alexander, J.~Beck, and W{.\,W.\,L.} Chen.
\newblock Geometric discrepancy theory and uniform distribution.
\newblock In J{.\,E.} Goodman and J.~O'Rourke, editors, {\em Handbook of
  Discrete and Computational Geometry}, chapter~10, pages 185--207. CRC Press
  LLC, Boca Raton, FL, 1997.

\bibitem[AE45]{a-pijd-45}
{T.~van} Aardenne-Ehrenfest.
\newblock Proof of the impossibility of a just distribution of an infinite
  sequence of points.
\newblock {\em Nederl. Akad. Wet., Proc.}, 48:266--271, 1945.
\newblock Also in {\em Indag. Math. } 7, 71-76 (1945).

\bibitem[AE49]{a-ijd-49}
{T.~van} Aardenne-Ehrenfest.
\newblock On the impossibility of a just distribution.
\newblock {\em Nederl. Akad. Wet., Proc.}, 52:734--739, 1949.
\newblock Also in {\em Indag. Math. } 11, 264-269 (1949).

\bibitem[Bal97]{Ball-ln}
K.~Ball.
\newblock An elementary introduction to modern convex geometry.
\newblock In S.~Levi, editor, {\em Flavors of Geometry (MSRI Publications vol.
  31)}, pages 1--58. Cambridge University Press, Cambridge, 1997.

\bibitem[Bal01]{Ball01-survey}
Keith Ball.
\newblock Convex geometry and functional analysis.
\newblock In {\em Handbook of the geometry of {B}anach spaces, {V}ol. {I}},
  pages 161--194. North-Holland, Amsterdam, 2001.

\bibitem[Ban98]{banasz98}
W.~Banaszczyk.
\newblock Balancing vectors and {G}aussian measures of $n$-dimensional convex
  bodies.
\newblock {\em Random Structures and Algorithms}, 12(4):351--360, 1998.

\bibitem[Ban10]{bansal-di}
N.~Bansal.
\newblock Constructive algorithms for discrepancy minimization.
\newblock \url{http://arxiv.org/abs/1002.2259}, also in \emph{FOCS'10: Proc.
  51st IEEE Symposium on Foundations of Computer Science}, pages 3--10, 2010.

\bibitem[BC87]{bc-id-87}
J.~Beck and W{.\,W.\,L.} Chen.
\newblock {\em Irregularities of {D}istribution}.
\newblock Cambridge University Press, Cambridge, 1987.

\bibitem[Bec81]{b-btcfs-81}
J.~Beck.
\newblock Balanced two-colorings of finite sets in the square. {I.}
\newblock {\em Combinatorica}, 1:327--335, 1981.

\bibitem[Bec89a]{b-btcfs-89}
J.~Beck.
\newblock Balanced two-colorings of finite sets in the cube.
\newblock {\em Discrete Mathematics}, 73:13--25, 1989.

\bibitem[Bec89b]{b-tdaet-89}
J.~Beck.
\newblock A two-dimensional van {Aardenne-Ehrenfest } theorem in irregularities
  of distribution.
\newblock {\em Compositio Math.}, 72:269--339, 1989.

\bibitem[Bha97]{Bhatia-MA}
Rajendra Bhatia.
\newblock {\em Matrix analysis}, volume 169 of {\em Graduate Texts in
  Mathematics}.
\newblock Springer-Verlag, New York, 1997.

\bibitem[BL08]{bl-sbi3D}
D.~Bilyk and M{.\,T.} Lacey.
\newblock On the small ball inequality in three dimensions.
\newblock {\em Duke Math. J.}, 143(1):81--115, 2008.

\bibitem[BLV08]{blv-sbi}
D.~Bilyk, M{.\,T.} Lacey, and A.~Vagharshakyan.
\newblock {On the small ball inequality in all dimensions}.
\newblock {\em J. Funct. Anal.}, 254(9):2470--2502, 2008.

\bibitem[Boh90]{b-odtp-90}
G.~Bohus.
\newblock On the discrepancy of 3 permutations.
\newblock {\em Random Struct. Algo.}, 1:215--220, 1990.

\bibitem[BS95]{bs-dt-95}
J.~Beck and V.~{S\'os}.
\newblock Discrepancy theory.
\newblock In {\em Handbook of Combinatorics}, pages 1405--1446. North-Holland,
  Amsterdam, 1995.

\bibitem[Cha00]{Chazelle-book}
B.~Chazelle.
\newblock {\em The Discrepancy Method}.
\newblock Cambridge University Press, Cambridge, 2000.

\bibitem[Che80]{c-id-81}
W{.\,W.\,L.} Chen.
\newblock On irregularities of distribution.
\newblock {\em Mathematika}, 27:153--170, 1980.

\bibitem[CL01a]{chl-tbhd}
B.~Chazelle and A.~Lvov.
\newblock {A trace bound for the hereditary discrepancy}.
\newblock {\em Discrete Comput. Geom.}, 26(2):221--231, 2001.

\bibitem[CL01b]{chaz-lvov-boxes}
B.~Chazelle and A.~Lvov.
\newblock {The discrepancy of boxes in higher dimension}.
\newblock {\em Discrete Comput. Geom.}, 25(4):519--524, 2001.

\bibitem[CNN11]{charikar-al-disc-inapprox}
M.~Charikar, A.~Newman, and A.~Nikolov.
\newblock Tight hardness results for minimizing discrepancy.
\newblock In {\em Proc. 22nd Annual ACM-SIAM Symposium on Discrete Algorithms
  (SODA), San Francisco, California, USA}, pages 1607--1614, 2011.

\bibitem[Cor35a]{c-vI-35}
{J.~G.~van~der} Corput.
\newblock Verteilungsfunktionen {I}.
\newblock {\em Akad. Wetensch. Amsterdam, Proc.}, 38:813--821, 1935.

\bibitem[Cor35b]{c-vII-35}
{J.~G.~van~der} Corput.
\newblock Verteilungsfunktionen {II}.
\newblock {\em Akad. Wetensch. Amsterdam, Proc.}, 38:1058--1066, 1935.

\bibitem[Dav56]{d-nid-56}
H.~Davenport.
\newblock Note on irregularities of distribution.
\newblock {\em Mathematika}, 3:131--135, 1956.

\bibitem[DSW04]{Doerr-al-productdisc}
B.~Doerr, A.~Srivastav, and P.~Wehr.
\newblock Discrepancy of {C}artesian products of arithmetic progressions.
\newblock {\em Electron. J. Combin.}, 11:Research Paper 5, 16 pp. (electronic),
  2004.

\bibitem[DT97]{drm-ti}
M.~Drmota and R{.\,F.} Tichy.
\newblock {\em Sequences, discrepancies and applications (Lecture Notes in
  Mathematics 1651)}.
\newblock Springer-Verlag, Berlin etc., 1997.

\bibitem[Fre82]{Fredman82}
Michael~L. Fredman.
\newblock The complexity of maintaining an array and computing its partial
  sums.
\newblock {\em J. {ACM}}, 29(1):250--260, 1982.

\bibitem[FSS01]{ForsterSS01}
J{\"u}rgen Forster, Niels Schmitt, and Hans~Ulrich Simon.
\newblock Estimating the optimal margins of embeddings in {E}uclidean half
  spaces.
\newblock In {\em Computational learning theory ({A}msterdam, 2001)}, volume
  2111 of {\em Lecture Notes in Comput. Sci}, pages 402--415. Springer, Berlin,
  2001.

\bibitem[GLS88]{gls-gaco-88}
M.~Gr{\"o}tschel, L.~Lov{\'a}sz, and A.~Schrijver.
\newblock {\em Geometric Algorithms and Combinatorial Optimization}, volume~2
  of {\em Algorithms and Combinatorics}.
\newblock Springer-Verlag, Berlin etc., 1988.
\newblock 2nd edition 1993.

\bibitem[Hal60]{h-ecqsp-60}
J{.\,H.} Halton.
\newblock On the efficiency of certain quasi-random sequences of points in
  evaluating multi-dimensional integrals.
\newblock {\em Numer. Math.}, 2:84--90, 1960.

\bibitem[Ham60]{h-mcmsm-60}
J{.\,M.} Hammersley.
\newblock Monte {C}arlo methods for solving multivariable problems.
\newblock {\em Ann. New York Acad. Sci.}, 86:844--874, 1960.

\bibitem[Lar14]{larsen14}
K.~G. Larsen.
\newblock On range searching in the group model and combinatorial discrepancy.
\newblock {\em SIAM Journal on Computing}, 43(2):673--686, 2014.

\bibitem[LMSS07]{LinialMSS07-signmatrices}
Nati Linial, Shahar Mendelson, Gideon Schechtman, and Adi Shraibman.
\newblock Complexity measures of sign matrices.
\newblock {\em Combinatorica}, 27(4):439--463, 2007.

\bibitem[LS09a]{LinialS09-learning}
Nati Linial and Adi Shraibman.
\newblock Learning complexity vs. communication complexity.
\newblock {\em Combin. Probab. Comput.}, 18(1-2):227--245, 2009.

\bibitem[LS09b]{LinialS09-CC}
Nati Linial and Adi Shraibman.
\newblock Lower bounds in communication complexity based on factorization
  norms.
\newblock {\em Random Structures Algorithms}, 34(3):368--394, 2009.

\bibitem[LS{\v{S}}08]{LeeSS08}
Troy Lee, Adi Shraibman, and Robert {\v{S}}palek.
\newblock A direct product theorem for discrepancy.
\newblock In {\em Proceedings of the 23rd Annual {IEEE} Conference on
  Computational Complexity, {CCC} 2008, 23-26 June 2008, College Park,
  Maryland, {USA}}, pages 71--80. {IEEE} Computer Society, 2008.

\bibitem[LSV86]{lsv-dssm-86}
L.~Lov\'asz, J.~Spencer, and K.~Vesztergombi.
\newblock Discrepancy of set-systems and matrices.
\newblock {\em European J. Combin.}, 7:151--160, 1986.

\bibitem[Mat98]{m-beckfilp}
J.~Matou{\v{s}}ek.
\newblock {An $L_p$ version of the Beck-Fiala conjecture}.
\newblock {\em European J. Combinatorics}, 19:175--182, 1998.

\bibitem[Mat99]{m-dipol-97}
J.~Matou\v{s}ek.
\newblock On the discrepancy for boxes and polytopes.
\newblock {\em Monatsh. Math.}, 127(4):325--336, 1999.

\bibitem[Mat10]{m-gd}
J.~Matou\v{s}ek.
\newblock {\em Geometric Discrepancy (An Illustrated Guide), 2nd printing}.
\newblock Springer-Verlag, Berlin, 2010.

\bibitem[Mat13]{almosttight}
J.~Matou{\v{s}}ek.
\newblock The determinant bound for discrepancy is almost tight.
\newblock {\em Proc. Amer. Math. Soc.}, 141(2):451--460, 2013.

\bibitem[MN12]{MN-stoc12}
S.~Muthukrishnan and A.~Nikolov.
\newblock Optimal private halfspace counting via discrepancy.
\newblock In {\em STOC '12: Proceedings of the 44th symposium on Theory of
  Computing}, pages 1285--1292, New York, NY, USA, 2012. ACM.

\bibitem[MN14]{e8-tusnady}
J.~Matou\v{s}ek and A.~Nikolov.
\newblock Combinatorial discrepancy for boxes via the ellipsoid-infinity norm.
\newblock Preprint at arXiv:1408.1376, to appear in SoCG 15 as "Combinatorial
  discrepancy for boxes via the $\gamma_2$ norm", 2014.

\bibitem[MS96]{ms-dap-96}
J.~Matou\v{s}ek and J.~Spencer.
\newblock Discrepancy in arithmetic progressions.
\newblock {\em J. Amer. Math. Soc.}, 9:195--204, 1996.

\bibitem[NNN12]{newman2012beck}
A.~Newman, O.~Neiman, and A.~Nikolov.
\newblock Beck's three permutations conjecture: A counterexample and some
  consequences.
\newblock In {\em Proc. 53rd Annual IEEE Symposium on Foundations of Computer
  Science (FOCS)}, pages 253--262, 2012.

\bibitem[NT13]{NT-HAP}
A.~Nikolov and K.~Talwar.
\newblock On the hereditary discrepancy of homogeneous arithmetic progressions.
\newblock {\em Proc. Amer. Math. Soc.}, 2013.
\newblock To appear. Preprint at arXiv:1309:6034.

\bibitem[NT15]{NT}
A.~Nikolov and K.~Talwar.
\newblock Approximating hereditary discrepancy via small width ellipsoids.
\newblock In {\em Proc. 26th {ACM-SIAM} Symposium on Discrete Algorithms},
  pages 324--336. {SIAM}, 2015.

\bibitem[NTZ13]{NTZ-diffprivacy}
A.~Nikolov, K.~Talwar, and Li~Zhang.
\newblock The geometry of differential privacy: the sparse and approximate
  cases.
\newblock In {\em Proc. 45th ACM Symposium on Theory of Computing (STOC), Palo
  Alto, California, USA}, pages 351--360, 2013.
\newblock Full version to appear in SIAM Journal on Computing as The Geometry
  of Differential Privacy: the Small Database and Approximate Cases.

\bibitem[P{\'a}l10]{palvo-wedges}
D.~P{\'a}lv{\"o}lgyi.
\newblock {Indecomposable coverings with concave polygons}.
\newblock {\em Discrete Comput. Geom.}, 44(3):577--588, 2010.

\bibitem[Rot54]{r-id-54}
K{.\,F.} Roth.
\newblock On irregularities of distribution.
\newblock {\em Mathematika}, 1:73--79, 1954.

\bibitem[Rot64]{KFR}
K{.\,F.} Roth.
\newblock Remark concerning integer sequences.
\newblock {\em Acta Arith.}, 9:257--260, 1964.

\bibitem[Rot80]{r-idIV-80}
K{.\,F.} Roth.
\newblock On irregularities of distribution {IV}.
\newblock {\em Acta Arith.}, 37:67--75, 1980.

\bibitem[Rot14]{Rothvoss14-convex}
Thomas Rothvo{\ss}.
\newblock Constructive discrepancy minimization for convex sets.
\newblock {\em CoRR}, abs/1404.0339, 2014.
\newblock To Appear in FOCS 2014.

\bibitem[Sch72]{s-idVII-72}
W{.\,M.} Schmidt.
\newblock On irregularities of distribution {VII}.
\newblock {\em Acta Arith.}, 21:45--50, 1972.

\bibitem[See93]{combinElli}
A.~Seeger.
\newblock Calculus rules for combinations of ellipsoids and applications.
\newblock {\em Bull. Australian Math. Soc.}, 47(01):1--12, 1993.

\bibitem[SM14]{StraMacNama}
G.~Strang and S.~MacNamara.
\newblock Functions of difference matrices are {T}oeplitz plus {H}ankel.
\newblock {\em SIAM Review}, 2014.
\newblock To appear.

\bibitem[Spe87]{s-tlpm-87}
J.~Spencer.
\newblock {\em Ten Lectures on the Probabilistic Method}.
\newblock CBMS-NSF. SIAM, Philadelphia, PA, 1987.

\bibitem[Sri97]{s-idbsm-96}
A.~Srinivasan.
\newblock Improving the discrepancy bound for sparse matrices: better
  approximations for sparse lattice approximation problems.
\newblock In {\em Proc. 8th ACM-SIAM Symposium on Discrete Algorithms}, pages
  692--701, 1997.

\bibitem[TJ89]{Tomczak-J-book}
Nicole Tomczak-Jaegermann.
\newblock {\em Banach-{M}azur distances and finite-dimensional operator
  ideals}, volume~38 of {\em Pitman Monographs and Surveys in Pure and Applied
  Mathematics}.
\newblock Longman Scientific \& Technical, Harlow; copublished in the United
  States with John Wiley \& Sons, Inc., New York, 1989.

\end{thebibliography}

\end{document}